\documentclass[a4paper,reqno,11pt]{amsart}

\usepackage{amsmath,amsfonts,amsthm,amssymb,euler,newpxtext,latexsym,mathrsfs,mathtools,enumerate}

\usepackage[dvipsnames]{xcolor}

\usepackage{tikz-cd,tikz-3dplot}

\usepackage{eqparbox}

\usepackage[colorlinks=true,linkcolor=blue,citecolor=red,urlcolor=teal,bookmarksdepth=subsection,final]{hyperref}

\DeclareMathOperator{\prob}{\mathcal{P}}
\DeclareMathOperator{\supp}{\mathrm{supp}}
\DeclareMathOperator*{\esssup}{\pi\text{-}ess\,sup}

\newcommand{\cM}{\mathcal{M}}
\newcommand{\nn}{\mathbb{N}}
\newcommand{\zz}{\mathbb{Z}}
\newcommand{\rr}{\mathbb{R}}
\newcommand{\cc}{\mathbb{C}}
\newcommand{\ff}{\mathbb{F}}
\newcommand{\cs}{\mathrm{C}^*}
\newcommand{\js}{\mathcal{Z}}

\newcommand{\eps}{\varepsilon}
\newcommand{\fp}{\mathfrak{p}}

\newcommand{\id}{\mathrm{id}}

\makeatletter
\@namedef{subjclassname@2020}{%
  \textup{2020} Mathematics Subject Classification}
\makeatother

\newenvironment{items}{\begin{list} {$\boldsymbol{\cdot}$} {\setlength{\leftmargin}{0.5cm}}}{\end{list}}

\newtheorem {theorem}{Theorem}[section]
\newtheorem {lemma}[theorem]{Lemma}
\newtheorem {proposition}[theorem]{Proposition}
\newtheorem {corollary}[theorem]{Corollary}

\newtheorem {thm}{Theorem}
\newtheorem {cor}[thm]{Corollary}

\theoremstyle {definition}

\newtheorem {example}[theorem]{Example}
\newtheorem {definition}[theorem]{Definition}
\newtheorem {remark}[theorem]{Remark}

\newtheorem {notation}[theorem]{Notation}

\numberwithin{equation}{section}

\title{A $\mathrm{C}^*$-algebraic Hoffman--Wielandt theorem}

\author[B.~Jacelon]{Bhishan Jacelon}
\address[B.~Jacelon]{
Institute of Mathematics of the Czech Academy of Sciences\\ \v{Z}itn\'{a} 25\\110 00 Praha 1\\Czech Republic}
\email{\href{mailto:jacelon@math.cas.cz}{jacelon@math.cas.cz}}


\begin{document}

\begin{abstract} 
We observe that the $2$-norm distance $d_{U,2}$ between the unitary orbits of normal elements in a $\mathrm{II}_1$ factor $\mathcal{M}$ is equal to the $2$-Wasserstein distance between the spectral measures induced by the trace $\tau_\mathcal{M}$. Using classification and optimal transport theory, we deduce an analogous $2$-norm equation for normal operators $x$ and $y$ in simple, separable, unital, nuclear, $\mathcal{Z}$-stable $\mathrm{C}^*$-algebras that are either monotracial, or real rank zero with finitely many extremal traces, provided that $\sigma(x)=\sigma(y)$ is convex. Consequently, $d_{U,2}$ equips the set of approximate unitary equivalence classes of contractive normal elements of $\mathcal{M}$ with the structure of a compact length space. The same is true of the set of equivalence classes of embeddings into the Jiang--Su algebra $\mathcal{Z}$ of classifiable tracial $2$-Wasserstein spaces over compact, convex planar domains.
\end{abstract}

\maketitle

\section*{Introduction} \label{section:intro}

This article is devoted to measure-theoretic computations of distances between unitary orbits of normal elements in $\cs$-algebras. The origin of this research programme is Weyl's theorem \cite[Satz 1]{Weyl:1912aa}, which implies that for self-adjoint matrices $x,y\in M_n$, the distance
\[
d_U(x,y)=\inf_{u\in U(M_n)}\|x-uyu^*\|
\]
between their unitary orbits is equal to the optimal matching distance
\[
\delta(x,y)=\inf_{\sigma\in S_n} \max_{1\le i\le n} |\alpha_i-\beta_{\sigma(i)}|
\]
between their eigenvalues. On the real line, this latter distance is attained at the identity permutation once the eigenvalues of $x$ and $y$ are listed in increasing order  $\alpha_1\le\dots\le\alpha_n$ and $\beta_1\le\dots\le\beta_n$. Lacking such a canonical ordering for more general (multi)sets of complex numbers, it is not \emph{a priori} clear whether the equality $d_U(x,y)=\delta(x,y)$ should hold for other classes of normal matrices. Since the inequality $d_U(x,y)\le\delta(x,y)$ is an easy consequence of diagonalisation, what is at issue is the reverse inequality $\delta(x,y)\le d_U(x,y)$. That is the substance of Weyl's theorem for self-adjoint matrices, and Bhatia and Davis's later theorem for unitaries \cite{Bhatia:1984aa}. The inequality is in fact known not to hold in general, even in $M_3$ \cite{Holbrook:1992aa}, although there is a universal constant $c$ such that  $\delta(x,y)\le c\cdot d_U(x,y)$ for normal elements $x,y$ in any $M_n$ \cite{Bhatia:1983aa} and indeed in any semifinite von Neumann algebra \cite{Davidson:1986aa,Hiai:1989aa}.

For tracial von Neumann or $\cs$-algebras, one can interpret the optimal matching distance $\delta$ as the $\infty$-Wasserstein distance \eqref{eqn:winfty} between spectral measures, and then look for classes of normal elements for which $d_U=W_\infty$ (see, for example, \cite{Hiai:1989aa,Hu:2015aa,Cheong:2015aa,Jacelon:2014aa,Jacelon:2021wa}). In \cite{Jacelon:2021vc,Jacelon:2025ab}, a case is made for broadening the scope of this analysis to include the full family of Wasserstein metrics $(W_\fp)_{\fp\in[1,\infty]}$ \eqref{eqn:transfer}. In light of the Kantorovich--Rubinstein formula for the $1$-Wasserstein distance between probability measures on a compact metric space $(X,\rho)$, namely
\[
W_1(\mu,\nu) = \sup\left\{\left|\int_Xf\,d\mu - \int_Xf\,d\nu\right| \mid f\colon X\to\rr \:\text{is $1$-Lipschitz}\right\},
\]
the $\fp=1$ analysis is intimately tied to the theory of tracial analogues \cite{Jacelon:2024aa} of Rieffel's compact quantum metric spaces \cite{Rieffel:1999aa,Rieffel:2004aa}. The impetus for this paper's emphasis on $\fp=2$ includes the foundational articles of Sturm \cite{Sturm:2006aa,Sturm:2006ab}, Lott and Villani \cite{Lott:2009aa} and Biane and Voiculescu \cite{Biane:2001aa}, subsequent work of Duvenhage \cite{Duvenhage:2022aa} and Gangbo, Jekel, Nam and Shlyakhtenko \cite{Gangbo:2022aa}, and recent advancements (summarised in Theorem~\ref{thm:hadwin} and Remark~\ref{rem:hadwin}) in the theory of unitary orbits in $\mathrm{II}_1$ factors due to Ding, Hadwin, Li and Liu \cite{Ding:2005aa,Li:2023aa}, and Hua and White \cite{Hua:2026aa}.

Our goal is to find classes of tracial $\cs$-algebras $A$ and normal elements $x,y\in A$ for which the $2$-norm distance $d_{U,2}(x,y)$ \eqref{eqn:dup} is equal to the $2$-Wasserstein distance $W_2(x,y)$ \eqref{eqn:wp}. When $A=M_n$ is a matrix algebra, the equality $d_{U,2}(x,y)=W_2(x,y)$ holds for \emph{arbitrary} normal elements $x,y\in A$. This is a consequence of the  inequality
\[
W_2(x,y) \le \|x-y\|_2
\]
proved by Hoffman and Wielandt in \cite{Hoffman:1953aa}. Equality also holds for Hilbert--Schmidt operators on infinite-dimensional Hilbert space \cite{Bhatia:1994aa}. Here, we observe the same for $\mathrm{II}_1$ factors (see Corollary~\ref{cor:hw}). We note that Theorem~\ref{thm:a} may be known to experts, certainly in the case of self-adjoint operators $x,y\in\cM$ (see \cite[Theorem 4.3]{Hiai:1989aa} and \cite[Theorem 1.5]{Biane:2001aa}). 

\begin{thm} \label{thm:a}
If $\cM$ is a $\mathrm{II}_1$ factor and $x,y\in\cM$ are normal, then $d_{U,2}(x,y)=W_2(x,y)$.
\end{thm}

There are a few available strategies for facilitating the move from von Neumann algebras to tracial $\cs$-algebras. For self-adjoint operators in $\cs$-algebras of real rank zero (meaning that the invertible self-adjoint elements are dense among all self-adjoint elements), one can compute $W_\infty$ as the distance between eigenvalue functions and proceed as in \cite{Sunder:1992aa} or \cite[Theorem 5.1]{Skoufranis:2016aa}. For more general classes of normal elements, there have been two different approaches. The first approach (more fundamental, and closer in spirit to the work of Hiai and Nakamura) uses Riesz interpolation in conjunction with weak unperforation, which are properties of the $K_0$-group that in particular hold when $A$ has real rank zero, stable rank one and tensorially absorbs the Jiang--Su algebra $\js$ (see \cite{Zhang:1990aa,Gong:2000kq,Rordam:2004kq}). This method establishes an analogue of the $\fp=\infty$ version of Theorem~\ref{thm:hn} (one half of the proof of Theorem~\ref{thm:a}) without restrictions on spectra. See especially \cite{Hu:2015aa} (in particular, \cite[Theorem 3.6]{Hu:2015aa}), as well as \cite{Elliott:2021aa,Wang:2023aa,Wang:2026aa,Wang:2025aa}.

The second approach, developed in \cite{Jacelon:2021wa}, combines classification and optimal transport theory. In Section~\ref{section:cstar}, we revisit, clarify and simplify some of this theory (see in particular Theorem~\ref{thm:infinitycircle} and Theorem~\ref{thm:lowerbound}), and we adapt it to $\fp<\infty$. Here, our focus shifts from normal elements to embeddings $\varphi \colon B \to A$ for certain commutative (or noncommutative \cite{Jacelon:2025ab}) domains $B$, and certain simple, separable, unital, nuclear, $\js$-stable codomains $A$. The idea is to solve an optimal transport problem in $B$ (see Theorem~\ref{thm:convex}), and then use the classification of embeddings \cite{Carrion:wz}  to convert the optimal transport automorphism of $B$ into an optimal conjugating unitary in $A$ (see Theorem~\ref{thm:transport}). The approach of \cite{Hu:2015aa}, in which $B$ is the $\cs$-algebra of continuous functions on the spectrum of a normal operator, is to directly solve a noncommutative transport problem in the $\cs$-algebra $A$ using the structural properties discussed above. Our technique needs additional assumptions on $A$ and $B$. Namely, we require either that $A$ has a unique trace and that $K_1(B)=0$, or that $A$ has real rank zero and finitely many extremal tracial states. In both cases, the classifying invariant $\underline{K}T_u$ of \cite{Carrion:wz} reduces to tracial data on the $K$-theoretically localised components \eqref{eqn:emb} of the set of unital embeddings of $B$ into $A$ (see Theorem~\ref{thm:class1} and Theorem~\ref{thm:cstarmetric}). We can then extract an inequality of the form $d_{U,\fp} \le k \cdot W_\fp$ for a suitable `transport constant ' $k$ that depends on the geometry of the domain (see Definition~\ref{def:transport}). An application of Hoffman--Wielandt  (Theorem~\ref{thm:lowerbound}) gives us the reverse inequality and our main result for tracial $\cs$-algebras (see Corollary~\ref{cor:normal}).

\begin{thm} \label{thm:b}
Let $A$ be a simple, separable, unital, nuclear, $\js$-stable $\cs$-algebra that is either monotracial, or real rank zero with finitely many extremal tracial states. If $x,y\in A$ are normal elements such that $\sigma(x)=\sigma(y)$ is convex, then $d_{U,2}(x,y)=W_2(x,y)$.
\end{thm}

The utility of these theorems is that they provide insight into the geometry of the set of approximate unitary equivalence classes of normal elements (in the $\mathrm{II}_1$ factor setting) or of embeddings (in the setting of classifiable tracial $\cs$-algebras). In similar spirit to ongoing work of Toms \cite{Toms:2023aa,Toms:2024aa}, we are especially interested in the existence and characterisation of length-minimising geodesics between classes. These geodesics are very well understood in classical Wasserstein space: they are `displacement interpolations' corresponding to optimal dynamical transference plans that describe how mass is transported over time. Theorem~\ref{thm:a} affords us an operator-theoretic interpretation of this metric structure (see Theorem~\ref{thm:lengthspaces}).

\begin{cor} \label{cor:c}
The metric space $(\mathfrak{n}(\cM),d_{U,2})$ of approximate unitary equivalence classes of contractive normal elements in a $\mathrm{II}_1$ factor $\cM$ is isometric to the compact, strictly intrinsic length space $(\prob(\mathbb{D}),W_2)$ of Borel probability measures on the unit disc.
\end{cor}

A sobering fact to bear in mind about Wasserstein geodesics in $\prob(\mathbb{D})$, or indeed in $\prob(X)$ for any smooth convex body $X\subseteq\cc$, is that a displacement interpolation between fully supported measures may pass through measures whose supports are not full and indeed not even convex. See \cite[Theorem 1]{Santambrogio:2016aa}. This means that we cannot immediately use Corollary~\ref{cor:transport} to find geodesics in the space of equivalence classes of embeddings of $C(X)$ into a tracial $\cs$-algebra $A$. The issue does not arise if we replace $C(X)$ by a \emph{simple} $\cs$-algebra $B$ with point-like $K$-theory \eqref{eqn:pointlike} whose tracial state space $T(B)$ is affinely homeomorphic to $\prob(X)$. In this case, the unital $^*$-homomorphisms $B\to A$ are all embeddings, which are now parameterised up to approximate unitary equivalence by the entirety of $\prob(X)$ (see Theorem~\ref{thm:class1}). The tracial Wasserstein space $\js(X)$ constructed in \cite[Corollary 3.7]{Jacelon:2025ab} has the right $K$-theory and trace space, and is designed to have the right structure for a sensible adaptation of Theorem~\ref{thm:b} (see Corollary~\ref{cor:zx}).

\begin{cor} \label{cor:d}
Let $X$ be a nonempty compact, convex subset of the Euclidean plane, let $(\js(X),W_2)$ be the corresponding projectionless, classifiable tracial Wasserstein space, and let $A$ be a simple, separable, unital, nuclear, $\js$-stable $\cs$-algebra with a unique trace. Then, the metric space $([\mathrm{Hom}^1(\js(X),A)],\overrightarrow{d_{U,2}})$ of approximate unitary equivalence classes of unital $^*$-homomorphisms $\js(X)\to A$ is isometric to the compact, strictly intrinsic length space $(\prob(X),W_2)$ of Borel probability measures on $X$.
\end{cor}

Notably, Theorem~\ref{thm:b} and Corollary~\ref{cor:d} apply to monotracial classifiable $\cs$-algebras like the Jiang--Su algebra $A=\js$ that do not have real rank zero.

\subsection*{Organisation} In Section~\ref{section:prelim}, we recall the definitions of the Wasserstein metrics on measures and the tracial Schatten distances between unitary orbits of normal elements of $\cs$-algebras, and we provide a brief primer on classification theory. In Section~\ref{section:II1}, we use the language of optimal transport theory to prove Theorem~\ref{thm:a} and Corollary~\ref{cor:c}. Finally, in Section~\ref{section:cstar}, we use optimal transport and classification to prove Theorem~\ref{thm:b} and Corollary~\ref{cor:d}.

\subsection*{Acknowledgements} This work is in memory of Prof.\ Ben Garling, whose Part III course on Isoperimetry and Concentration of Measure was my introduction to optimal transport theory. I am also grateful to Shanshan Hua and David Jekel for valuable conversations on the subject matter of the paper, and to the anonymous referee for providing a substantial amount of helpful feedback that has greatly improved the quality of the present exposition. The research was supported by the Czech Science Foundation (GA\v{C}R) project 25-15444K and the Institute of Mathematics of the Czech Academy of Sciences (RVO: 67985840). Some reference searching was facilitated by the Gemini large language model integrated into Google Search.

\section{Preliminaries} \label{section:prelim}

\begin{notation} \label{notation}
We adopt the following notation and terminology throughout the article. In this discussion, $A$ is a unital $\cs$-algebra.
\begin{items}
\item We write $A_{sa}=\{a\in A \mid a^*=a\}$ for the set of self-adjoint elements of $A$, and $U(A)=\{u\in A \mid u^*u=uu^*=1\}$ for the unitary group of $A$. Given $a\in A$, we refer to the set $\{uau^* \mid u\in U(A)\}$ as the \emph{unitary orbit of $a$}.
\item We let $\sigma(a)$ denote the spectrum of an element $a\in A$. If $a$ is normal, that is, $a^*a=aa^*$, then $\varphi_a$ denotes the functional calculus $^*$-homomorphism $C(\sigma(a)) \to \cs(a,1) \subseteq A$ that maps the identity function $\id_{\sigma(a)}$ to $a$.
\item The absolute value $|a|$ of an element $a\in A$ is the element $(a^*a)^{\frac{1}{2}}$ computed via the functional calculus map $\varphi_{a^*a}$.
\item We say that $A$ has \emph{real rank zero} if the invertible elements of $A_{sa}$ are dense in $A_{sa}$, or equivalently, if the elements in $A_{sa}$ with finite spectrum are dense in $A_{sa}$ (see \cite[Theorem 2.6]{Brown:1991aa}).
\item By a \emph{trace} on $A$ we always mean a tracial state, that is, a unital, positive linear map $\tau\colon A \to \mathbb{C}$ such that $\tau(uau^*)=\tau(a)$ for every $a\in A$ and $u\in U(A)$. We write $T(A)$ for the simplex of tracial states on $A$, and $\partial_e(T(A))$ for the set of extreme points of $T(A)$. We say that $A$ is \emph{tracial} if $T(A)\ne\emptyset$. Given a unital $^*$-homomorphism $\varphi \colon B \to A$ between tracial $\cs$-algebras, we write $T(\varphi)$ for the induced continuous affine map $T(A) \to T(B)$, $\tau \mapsto \tau\circ\varphi$.
\item If $A=C(X)$ for some compact, Hausdorff space $X$, then we identify $T(A)$ with the set $\prob(X)$ of Borel probability measures on $X$. Given a measure $\mu$ and a continuous function $h$, $h_*\mu$ denotes the pushforward measure $\mu\circ h^{-1}$.
\end{items}
\end{notation}

\begin{definition} \label{def:wasserstein}
Let $(X,\rho)$ be a compact metric space and let $\fp\in[1,\infty]$. Given $\mu,\nu\in\prob(X)$, let $\Pi(\mu,\nu)$ denote the set of \emph{transference plans} between $\mu$ and $\nu$, that is, the set of Borel probability measures on $X\times Y$ with marginals $\mu$ and $\nu$. In other words, if $p_1,p_2\colon X\times X \to X$ denote the two coordinate projections, then $\pi\in\Pi(\mu,\nu)$ if and only if $(p_1)_*\pi=\mu$ and $(p_2)_*\pi=\nu$. For $\fp<\infty$, the \emph{$\fp$-Wasserstein distance} between $\mu$ and $\nu$ is
\begin{equation} \label{eqn:transfer}
W_\fp(\mu,\nu) = \inf_{\pi\in\Pi(\mu,\nu)} \left(\int_{X\times X} \rho(w,z)^\fp\, d\pi(w,z)\right)^\frac{1}{\fp}.
\end{equation}
The \emph{$\infty$-Wasserstein distance} between $\mu$ and $\nu$ is
\begin{equation} \label{eqn:winfty}
W_\infty(\mu,\nu) = \inf_{\pi\in\Pi(\mu,\nu)} \esssup_{(w,z)\in X\times X} \rho(w,z) = \inf\{r>0 \mid(\forall\, \text{ Borel }  U) \:\mu(U) \le \nu(U_r)\}
\end{equation}
where $U_r = \{x \in X \mid \rho(x,U) < r\}$.
\end{definition}

The next proposition records some well-known facts about Wasserstein metrics that are discussed in more detail in \cite{Givens:1984to} and \cite{Jacelon:2025ab}.

\begin{proposition} \label{prop:wassfacts}
Let $(X,\rho)$ be a compact metric space.
\begin{enumerate}
\item \label{it:wassfacts1} The Wasserstein metrics $W_\fp$ all agree with $\rho$ on $X\cong\partial_e\prob(X)$ and all except $W_\infty$ induce the weak$^*$-topology on $\prob(X)$.
\item \label{it:wassfacts2} For $\fp\in[1,\infty)$, $W_\fp^\fp$ is convex in the sense that
\[
W_\fp\left(\sum_{i=1}^n\lambda_i\mu_i,\sum_{i=1}^n\lambda_i\nu_i\right)^\fp \le \sum_{i=1}^n\lambda_iW_\fp(\mu_i,\nu_i)^\fp
\]
for every $n\in\nn$, $\mu_1,\dots,\mu_n,\nu_1,\dots,\nu_n\in\prob(X)$ and $\lambda_1,\dots,\lambda_n\in[0,1]$ with $\sum_{i=1}^n\lambda_i=1$.
\item \label{it:wassfacts3} The $\fp$-Wasserstein distances are monotone in the sense that for every $\mu,\nu\in\prob(X)$, $W_\fp(\mu,\nu)\le W_{\fp'}(\mu,\nu)$ whenever $\fp\le\fp'$, and $W_\infty(\mu,\nu) = \lim_{\fp\to\infty}W_\fp(\mu,\nu)$.
\end{enumerate}
\end{proposition}

\begin{definition} \label{def:schatten}
Let $A$ be a unital $\cs$-algebra with a faithful trace, and let $\fp\in[1,\infty]$. For $\fp<\infty$, the \emph{Schatten $\fp$-norm relative to $\tau\in T(A)$} is defined on $A$ by
\[
\|a\|_{\fp,\tau} = \tau(|a|^\fp)^{\frac{1}{\fp}}
\]
and the \emph{uniform tracial Schatten $\fp$-norm} on $A$ is
\[
\|a\|_\fp = \sup_{\tau\in T(A)}\|a\|_{\fp,\tau}.
\]
For $\fp=2$, this is the usual uniform tracial $2$-norm $\|a\|_2=\sup_\tau\tau(a^*a)^{\frac{1}{2}}$. For $p=\infty$, we set $\|a\|_\infty = \|a\|$. Given normal elements $x,y\in A$, we let $\sigma(x)$ and $\sigma(y)$ denote their spectra and $\mu_{\tau,x}$ and $\mu_{\tau,y}$ denote the Borel probability measures on $X:=
\sigma(x)\cup\sigma(y)$ induced by $\tau\in T(A)$. The \emph{$\fp$-norm unitary orbit distance relative to $\tau\in T(A)$} between $x$ and $y$ is
\[
d_{U,\fp,\tau}(x,y) = \inf_{u\in \mathcal{U}(A)} \|x - uyu^*\|_{\fp,\tau}
\]
and the \emph{(uniform) $\fp$-norm unitary orbit distance} between $x$ and $y$ is
\begin{equation} \label{eqn:dup}
d_{U,\fp}(x,y) = \inf_{u\in \mathcal{U}(A)} \|x - uyu^*\|_\fp.
\end{equation}
The \emph{$\fp$-Wasserstein distance relative to $\tau\in T(A)$} between $x$ and $y$ is
\[
W_{\fp,\tau}(x,y) = W_\fp(\mu_{\tau,x},\mu_{\tau,y})
\]
and the \emph{uniform $\fp$-Wasserstein distance} between $x$ and $y$ is
\begin{equation} \label{eqn:wp}
W_\fp(x,y) = \sup_{\tau\in T(A)} W_{\fp,\tau}(x,y).
\end{equation}
\end{definition}

\begin{proposition} \label{prop:metric}
Let $A$ be a unital $\cs$-algebra with a faithful trace $\tau\in T(A)$, and let $\fp\in[1,\infty]$. Then,
\begin{enumerate}
\item  \label{it:metric1} $\|\cdot\|_{\fp,\tau}$ and $\|\cdot\|_{\fp}$ are norms on $A$, and $\|\cdot\|_{\fp,\sigma}$ is a seminorm on $A$ for any $\sigma\in T(A)$;
\item \label{it:metric2} for every $a\in A$, $\sigma\in T(A)$ and $\fp\le\fp'\le\infty$, we have $\|a\|_{\fp,\sigma} \le \|a\|_{\fp',\sigma}$, $\|a\|_{\fp} \le \|a\|_{\fp'}$, and $\lim_{\fp'\to\infty}\|a\|_{\fp'}=\|a\|$;
\item \label{it:metric3} unitary conjugation is an isometric action of $U(A)$ on $(A,\|\cdot\|_{\fp,\sigma})$ for every $\sigma\in T(A)$, hence on $(A,\|\cdot\|_{\fp})$;
\item \label{it:metric4} $d_{U,\fp}$ (respectively, $d_{U,\fp,\sigma}$ for any $\sigma\in T(A)$) is a metric on the set of $\|\cdot\|_{\fp}$-closures (respectively, $\|\cdot\|_{\fp,\sigma}$-closures) of unitary orbits of elements in $A$.
\end{enumerate}
\end{proposition}

\begin{proof}
The only nontrivial assertion in \eqref{it:metric1} is the triangle inequality for $\|\cdot\|_{\fp,\sigma}$ for $\fp<\infty$ and $\sigma\in T(A)$. This is proved by Nelson in \cite[Section 3]{Nelson:1974aa} in the context of a von Neumann algebra equipped with a faithful, normal, semifinite trace. Let us see how this implies the $\cs$-algebraic version. Let  $\pi_\sigma\colon A \to \mathcal{B}(H_\sigma)$ be the Gelfand--Naimark--Segal (GNS) representation associated with $\sigma$. By construction, there is a distinguished unit cyclic vector $\xi_\sigma$ in the Hilbert space $H_\sigma$ such that $\sigma(a) = \langle \xi_\sigma,\pi_\sigma(a)\xi_\sigma \rangle$ for every $a\in A$. Then, $\tilde\sigma(x):= \langle \xi_\sigma, x\xi_\sigma \rangle$ is a faithful, normal, finite trace on the weak closure $M_\sigma$ of $\pi_\sigma(A)$ in $\mathcal{B}(H_\sigma)$ (see \cite[Proposition V.3.19]{Takesaki:1979wd}). Since $\tilde\sigma\circ\pi_\sigma=\sigma$ and $^*$-homomorphisms commute with the continuous functional calculus, we have
\[
\|\pi_\sigma(a)\|_{\fp,\tilde\sigma} = \tilde\sigma(|\pi_\sigma(a)|^\fp)^{\frac{1}{\fp}} = \tilde\sigma\circ\pi_\sigma(|a|^\fp)^{\frac{1}{\fp}} = \sigma(|a|^\fp)^{\frac{1}{\fp}} = \|a\|_{\fp,\sigma}.
\]
We therefore obtain the triangle inequality in $(A,\|\cdot\|_{\fp,\sigma})$ from Nelson's analysis in $(M_\sigma,\|\cdot\|_{\fp,\tilde\sigma})$. Taking the supremum over all $\sigma\in T(A)$ gives the triangle inequality for $(A,\|\cdot\|_{\fp})$. The assumed faithfulness of $\tau$ ensures that $\|\cdot\|_{\fp,\tau}$, hence $\|\cdot\|_{\fp}$, is a norm rather than a seminorm.

Identifying the trace $\sigma$ on $\cs(a^*a,1)$ with a Borel probability measure $\mu_\sigma$ on $C(\sigma(a^*a))$, \eqref{it:metric2} follows from the well-known monotonicity of the classical $L^\fp$ norms.

Next, \eqref{it:metric3} follows from the invariance of traces under unitary conjugation: given $\sigma\in T(A)$, $a\in A$ and $u\in U(A)$, we have
\[
\|uau^*\|_{\fp,\sigma} = \sigma(|uau^*|^\fp)^{\frac{1}{\fp}} = \sigma(u|a|^\fp u^*)^{\frac{1}{\fp}} = \sigma(|a|^\fp)^{\frac{1}{\fp}} = \|a\|_{\fp,\sigma}.
\]
Taking the supremum over all $\sigma\in T(A)$ gives the same equality for $\|\cdot\|_\fp$.

Finally, let us prove \eqref{it:metric4}. Noting that $\|x-uyu^*\|_{\fp} = \|u^*xu-y\|_{\fp}$ (by \eqref{it:metric3}), we see that $d_{U,\fp}$ is symmetric. Next, we will show that 
\begin{equation} \label{eqn:triangle}
d_{U,\fp}(x,z) \le d_{U,\fp}(x,y) + d_{U,\fp}(y,z)
\end{equation}
for every $x,y,z\in A$. This will not only demonstrate the triangle inequality, but also that $d_{U,\fp}$ is a well-defined, positive definite function on $\|\cdot\|_{\fp}$-closures of unitary orbits. Note indeed that $x$ is in the $\|\cdot\|_{\fp}$-closure of the unitary orbit of $y$ (in symbols, $x\sim_{\fp}y$) if and only if $d_{U,\fp}(x,y)=d_{U,\fp}(y,x)=0$. If $x\sim_{\fp}x'$ and $y\sim_{\fp}y'$, we obtain from \eqref{eqn:triangle} that
\[
d_{U,\fp}(x',y') \le d_{U,\fp}(x',x) + d_{U,\fp}(x,y) + d_{U,\fp}(y,y') = d_{U,\fp}(x,y).
\]
Symmetrically, $d_{U,\fp}(x,y) \le d_{U,\fp}(x',y')$, hence $d_{U,\fp}(x,y) = d_{U,\fp}(x',y')$ as required. To prove \eqref{eqn:triangle}, fix $x,y,z\in A$, let $\eps>0$, and let $u,v\in U(A)$ such that $\|x-uyu^*\|_{\fp}<d_{U,\fp}(x,y)+\eps$ and $\|y-vzv^*\|_{\fp}<d_{U,\fp}(y,z)+\eps$. Then, $uv\in U(A)$ satisfies
\begin{align*}
\|x-(uv)z(uv)^*\|_{\fp} \:&\overset{\mathclap{\scriptstyle\eqref{it:metric1}}}{\le}\: \|x-uyu^*\|_{\fp} + \|uyu^*-(uv)z(uv)^*\|_{\fp}\\
\:&\overset{\mathclap{\scriptstyle\eqref{it:metric3}}}{=}\: \|x-uyu^*\|_{\fp} + \|y-vzv^*\|_{\fp}\\
\:&<\: d_{U,\fp}(x,y) + d_{U,\fp}(y,z) + 2\eps.
\end{align*}
Since $\eps$ is arbitrary, we deduce that \eqref{eqn:triangle} holds. The argument for $d_{U,\fp,\sigma}$ is exactly the same, just with $\|\cdot\|_{\fp,\sigma}$ and $\sim_{\fp,\sigma}$ in place of $\|\cdot\|_\fp$ and $\sim_{\fp}$.
\end{proof}

\begin{example}[Finitely supported measures] \label{ex:finite}
A transference plan between two finitely supported measures $\mu=\sum_{i=1}^ms_i\delta_{\alpha_i}$ and $\nu=\sum_{j=1}^nt_j\delta_{\beta_j}$ can equivalently be described as an $m\times n$ matrix with entries $a_{ij}\in[0,1]$ that satisfy $\sum_{j=1}^na_{ij}=s_i$ for every $i$ and $\sum_{i=1}^ma_{ij}=t_j$ for every $j$. In the special case where $m=n$ and $s_i=t_i=\frac{1}{n}$ for every $i$, an \emph{optimal} transference plan can be found among permutation matrices, that is,
\begin{equation} \label{eqn:lp}
W_\fp\left(\frac{1}{n}\sum_{i=1}^n\delta_{\alpha_i},\frac{1}{n}\sum_{i=1}^n\delta_{\beta_i}\right) = \left(\inf_{\sigma\in S_n} \frac{1}{n} \sum_{i=1}^n \rho(\alpha_i,\beta_{\sigma(i)})^\fp\right)^\frac{1}{\fp}.
\end{equation}
For $\fp<\infty$, this latter fact follows from convexity of the map $\pi\mapsto\int \rho(x,y)^\fp\,d\pi(x,y)$, together with the Birkhoff--von Neumann theorem, which says that the extreme points of the set of bistochastic matrices are exactly the permutation matrices (see \cite{Birkhoff:1946aa}). A well-known consequence of Hall's marriage lemma is the analogous equation for $\fp=\infty$, that is,
\begin{equation} \label{eqn:linfty}
W_\infty\left(\frac{1}{n}\sum_{i=1}^n\delta_{\alpha_i},\frac{1}{n}\sum_{i=1}^n\delta_{\beta_i}\right) = \inf_{\sigma\in S_n} \max_{1\le i\le n}\rho(\alpha_i,\beta_{\sigma(i)}).
\end{equation}
Via diagonalisation, these equations imply that $d_{U,\fp}(x,y) \le W_\fp(x,y)$ for any $\fp\in[1,\infty]$ and any pair of normal matrices $x,y\in M_n$.  The reverse inequality for $\fp=2$ is proved in \cite{Hoffman:1953aa}. (See also the discussion in \cite[Section 1.5]{Gangbo:2022aa}.)
\end{example}

The relation $\sim_{\fp}$ appearing in the proof of Proposition~\ref{prop:metric}\eqref{it:metric4} is an equivalence relation known as \emph{approximate unitary equivalence in $\fp$-norm}. This terminology also refers to an equivalence relation on $^*$-homomorphisms, formulated below.

\begin{definition} \label{def:auequiv}
Let $A$ be a unital $\cs$-algebra with a faithful trace, let $\fp\in[1,\infty]$, and let $x,y \in A$. We say that  \emph{$x$ is approximately unitarily equivalent to $y$ in $\fp$-norm}, or simply that \emph{$x$ is  approximately unitarily equivalent to $y$} if $\fp=\infty$, if there is a sequence $(u_n)_{n=1}^\infty$ of unitaries in $A$ such that $\lim_{n\to\infty}\|x-u_nyu_n^*\|_\fp=0$. Given another $\cs$-algebra $B$ and $^*$-homomorphisms $\varphi,\psi\colon B\to A$, we say that \emph{$\varphi$ is approximately unitarily equivalent to $\psi$ in $\fp$-norm}, or simply that \emph{$\varphi$ is approximately unitarily equivalent to $\psi$} if $\fp=\infty$, if, for every $\varepsilon>0$ and every finite set $F\subseteq B$, there is a unitary $u\in A$ such that $\|\varphi(b)-u\psi(b)u^*\|_\fp<\varepsilon$ for every $b\in F$. Note that if $B$ is separable (in norm), this is equivalent to the existence of a sequence $(u_n)_{n=1}^\infty$ of unitaries in $B$ such that $\lim_{n\to\infty}\|\varphi(b)-u_n\psi(b)u_n^*\|_\fp=0$ for every $b\in B$.
\end{definition}

In Section~\ref{section:cstar}, we make frequent and crucial use of the classification theory of Carri\'{o}n, Gabe, Schafhauser, Tikuisis and White \cite{Carrion:wz} (though it should be noted, as extensively detailed in the introduction of \cite{Carrion:wz}, that this theory is based on decades of work by many mathematicians). Specifically, we appeal to \cite[Theorem B]{Carrion:wz} and \cite[Theorem 1.1]{Carrion:wz}, both of which pertain to the classification of approximate unitary equivalence classes of certain unital embeddings (that is, injective $^*$-homomorphisms) by an invariant $\underline{K}T_u$ consisting of $K$-theoretic and tracial data. In both cases, the allowed domains are unital, separable, nuclear $\cs$-algebras satisfying the universal coefficient theorem (UCT) of Rosenberg and Schochet \cite{Rosenberg:1987fp}. This applies to all domains that we consider in Section~\ref{section:cstar}. In particular, the UCT holds for commutative $\cs$-algebras, and more broadly for sequential inductive limits of $\cs$-algebras of type I (see \cite[22.3.5]{Blackadar:1998qf}).

In \cite[Theorem B]{Carrion:wz}, the allowed codomains of the classified embeddings are unital, simple, separable, nuclear, $\js$-stable $\cs$-algebras. Here, $\js$ is the Jiang--Su algebra constructed in \cite{Jiang:1999hb}, and a $\cs$-algebra $A$ is \emph{$\js$-stable} if $A\otimes\js \cong A$. In our applications of this theorem, the total invariant $\underline{K}T_u$ defined in \cite[Section 3.2]{Carrion:wz} reduces to traces and $K$-theory (rather than traces, total $K$-theory \cite[Section 2.3]{Carrion:wz} and Hausdorffised unitary algebraic $K_1$ \cite[Section 2.2]{Carrion:wz}). See Theorem~\ref{thm:class1}. In \cite[Theorem 1.1]{Carrion:wz}, the codomains are the \emph{sequence algebras} $A_\infty=\ell^\infty(A)/c_0(A)$ associated with unital, simple, separable, nuclear, $\js$-stable $\cs$-algebras $A$. We use this theorem to obtain a \emph{local} version of Theorem~\ref{thm:class1}\eqref{it:class1}. See Theorem~\ref{thm:class2}.

\begin{theorem} \label{thm:class1}
Let $A$ be a simple, separable, unital, nuclear, $\js$-stable, tracial $\cs$-algebra, and let $B$ be a unital, separable, nuclear, tracial $\cs$-algebra that satisfies the UCT.
\begin{enumerate}
\item \label{it:class1} If $K_*(B)$ is finitely generated and torsion free, and either $A$ has real rank zero or $K_1(B)=0$, then two unital embeddings $\varphi,\psi \colon B \to A$ are approximately unitarily equivalent if and only if $K_*(\varphi)=K_*(\psi)$ and $T(\varphi)=T(\psi)$.
\item \label{it:class2} If $A$ has a unique trace, and $B$ is simple with $(K_0(B),[1_B],K_1(B)) \cong (\zz,1,0)$, then there is a bijective correspondence between $T(B)$ and the set $[\mathrm{Hom}^1(B,A)]$ of approximate unitary equivalence classes of unital $^*$-homomorphisms $B\to A$.
\end{enumerate} 
\end{theorem}

\begin{proof}
The first claim \eqref{it:class1} is just a very special case of the uniqueness part of \cite[Theorem B]{Carrion:wz}. Let $\varphi,\psi \colon B \to A$ be unital embeddings. Since $K_*(B)$ is assumed to be a finitely generated free abelian group, it follows from the UCT that $\varphi$ and $\psi$ agree in total $K$-theory if and only if $K_*(\psi) = K_*(\varphi)$ (see the discussions on the $KL$ functor in \cite[Section 1.3.5]{Carrion:wz} and \cite[Section 2.4]{Rordam:2002yu}). The remaining assumptions on $A$ and $B$ ensure that the $\overline{K}_1^{\mathrm{alg}}$ component of $\underline{K}T_u$ also disappears. Here, $\overline{K}_1^{\mathrm{alg}}(B)$ is the Hausdorffised unitary algebraic $K_1$-group of a unital $\cs$-algebra $B$. By definition, $\overline{K}_1^{\mathrm{alg}}(B)$ is the quotient $U_\infty(B)/\overline{DU_\infty(B)}$, where $U_\infty(B)$ is the inductive limit of the unitary groups $U_n(B)$ of matrix algebras $M_n(B)$ over $B$, and $\overline{DU_\infty(B)}$ is the closure in the inductive limit topology of the union of the derived subgroups (that is, the subgroups generated by commutators) of the groups $U_n(B)$. It is related to $K$-theory and traces via the Thomsen exact sequence
\begin{equation} \label{eqn:thomsen}
\begin{tikzcd}[ampersand replacement=\&]
0 \arrow[r] \& \mathrm{Aff}(T(B))/\overline{\rho_B(K_0(B))} \arrow[r,"\lambda_B"] \& \overline{K_1}^{\mathrm{alg}}(B) \arrow[r,"\pi_B"] \& K_1(B) \arrow[r] \& 0.
\end{tikzcd}
\end{equation}
Here, $\mathrm{Aff}(T(B))$ is the set of weak$^*$-continuous affine functions $T(B) \to \rr$, and $\rho_B \colon K_0(B) \to \mathrm{Aff}(T(B))$ is the pairing homomorphism defined on projections $p\in M_n(B)$ by $\rho_B([p])(\tau) = \tau(p)$ (where traces $\tau\in T(B)$ are canonically extended to norm-$n$ tracial functionals on $M_n(B)$). If $A$ has real rank zero, then $\rho_A(K_0(A))$ is dense in $\mathrm{Aff}(T(A))$  (see \cite[Theorem 6.9.3]{Blackadar:1998qf}), so $\overline{K}_1^{\mathrm{alg}}(A) \cong K_1(A)$. If $K_1(B)=0$, then $\overline{K}_1^{\mathrm{alg}}(B) \cong \mathrm{Aff}(T(B))/\overline{\rho_{B}(K_0(B))}$. Since $\underline{K}T_u$-morphisms commute with the Thomsen sequence, in both cases we then automatically have that $\overline{K}_1^{\mathrm{alg}}(\varphi) = \overline{K}_1^{\mathrm{alg}}(\psi)$. Therefore, in both cases, $K_*(\psi) = K_*(\varphi)$ and $T(\varphi)=T(\psi)$ if and only if $\underline{K}T_u(\varphi) = \underline{K}T_u(\psi)$, which by \cite[Theorem B]{Carrion:wz} holds if and only if $\varphi$ and $\psi$ are approximately unitarily equivalent.

For the second claim \eqref{it:class2}, we need both the uniqueness and existence parts of \cite[Theorem B]{Carrion:wz}. Since $A$ has a unique trace $\tau_A$, the morphisms $\gamma \colon \mathrm{Aff}(T(B)) \to \mathrm{Aff}(T(A)) \cong \rr$ are parameterised by $T(B)$. Namely, $\tau\in T(B)$ corresponds to $\gamma_\tau \colon f \mapsto f(\tau)$. Since $K_0(B)$ is generated by $[1_B]$, there is also a unique morphism $\alpha_0 \colon K_0(B) \to K_0(A)$ sending $[1_B]$ to $[1_A]$, which is moreover automatically compatible with $\gamma_\tau$: identifying $\mathrm{Aff}(T(A))$ with $\rr$, we have $\gamma_\tau(\rho_B([1_B])) = \tau(1_B) = 1 = \tau_A(1_A) = \rho_A(\alpha_0([1_B]))$. There is of course only the trivial morphism $\alpha_1$ from $K_1(B)=0$ to $K_1(A)$. As in \eqref{it:class1}, there is a unique extension of $\alpha_*$ to total $K$-theory, and a unique $\overline{K}_1^{\mathrm{alg}}$-morphism that is compatible with the Thomsen sequence \eqref{eqn:thomsen}. Putting this all together, we see that the $\underline{K}T_u$-morphisms from $B$ to $A$ are in bijective correspondence with the traces on $B$. Since $B$ is simple, every trace $\tau\in T(B)$ is faithful, and every $^*$-homomorphism $B\to A$ is an embedding. Thus, \eqref{it:class2} follows from \cite[Theorem B]{Carrion:wz}.
\end{proof}

\begin{theorem} \label{thm:class2}
Let $A$ be a simple, separable, unital, nuclear, $\js$-stable, tracial $\cs$-algebra, and let $B$ be a unital, separable, nuclear, tracial $\cs$-algebra that satisfies the UCT and has finitely generated, torsion-free $K$-theory $K_*(B)$. Suppose further either that $A$ has real rank zero or that $K_1(B)=0$. Let $\rho$ be a metric on $T(B)$ that induces the weak$^*$-topology. Then, for every unital embedding $\varphi \colon B\to A$, every compact set $F\subseteq B$ and every $\eps>0$, there exists $\delta>0$ such that, for every unital embedding $\psi \colon B\to A$ with $K_*(\varphi)=K_*(\psi)$ and $\sup_{\tau\in T(A)}\rho(\tau\circ\varphi,\tau\circ\psi)<\delta$, there is a unitary $u\in A$ such that $\|\varphi(b)-u\psi(b)u^*\|<\eps$ for every $b\in F$.
\end{theorem}

\begin{proof}
Suppose that the claim is not true. Then, there exists a unital embedding $\varphi \colon B\to A$, a compact set $F\subseteq B$, an $\eps>0$, and a sequence of unital embeddings $\psi_n \colon B \to A$ such that $K_*(\varphi)=K_*(\psi_n)$ and $\lim_{n\to\infty}\sup_{\tau\in T(A)}\rho(\tau\circ\varphi,\tau\circ\psi_n)=0$, but $\max_{b\in F}\|\varphi(b)-u\psi_n(b)u^*\| \ge \eps$ for every $u\in U(A)$. The idea is to define $\Phi,\Psi \colon B \to A_\infty$ by $\Phi(b)=[(\varphi(b),\varphi(b),\dots)]$ and $\Psi(b)=[(\psi_1(b),\psi_2(b),\dots)]$, observe that $\underline{K}T_u(\Phi) = \underline{K}T_u(\Psi)$, then appeal to \cite[Theorem 1.1]{Carrion:wz} to obtain a contradiction. Indeed, noting that 
\begin{items}
\item $\Phi$ and $\Psi$ are \emph{full} embeddings (by \eqref{eqn:omega} and \cite[Lemma 4.10 and Proposition 4.17]{Carrion:wz}), and
\item unitaries in $A_\infty$ can be lifted to unitaries in $\ell^\infty(A)$ (see \cite[Lemma 6.2.4]{Rordam:2002yu}),
\end{items}
this theorem would provide a sequence $(u_n)_{n=1}^\infty$ of unitaries in $A$ such that $\limsup_{n\to\infty}\|\varphi(b)-u_n\psi_n(b)u_n^*\|=0$ for every $b\in B$. Taking a finite $\frac{\eps}{2}$-net for $F$, we would then deduce that $\max_{b\in F}\|\varphi(b)-u_n\psi_n(b)u_n^*\| < \eps$ for all sufficiently large $n$, which is the desired contradiction.

It remains to show that $\underline{K}T_u(\Phi) = \underline{K}T_u(\Psi)$. Since $K_*(\varphi)=K_*(\psi_n)$ for every $n\in\nn$, we have $K_*(\Phi)=K_*(\Psi)$. As in Theorem~\ref{thm:class1}, the assumption that $K_*(B)$ is finitely generated and torsion free then implies that $\Phi$ and $\Psi$ agree in total $K$-theory. Since $\rho$ is assumed to induce the weak$^*$-topology on $T(B)$, and $\lim_{n\to\infty}\sup_{\tau\in T(A)}\rho(\tau\circ\varphi,\tau\circ\psi_n)=0$, we have that $\lim_{n\to\infty}\sup_{\tau\in T(A)}|\tau(\varphi(b))-\tau(\psi_n(b))| = 0$ for every $b\in B$. As observed in \cite[Proposition 4.16]{Carrion:wz}, the trace space $T(A_\infty)$ is the weak$^*$-closed convex hull of the set of limit traces, that is, those of the form $\tau([(a_n)_{n=1}^\infty]) = \lim_{n\to\omega}\tau_n(a_n)$ for some sequence $(\tau_n)_{n=1}^\infty$ in $T(A)$ and some free ultrafilter $\omega$ on $\nn$. For such a limit trace $\tau$, and any $b\in B$, we then have
\begin{equation} \label{eqn:omega}
\tau(\Phi(b)) = \lim_{n\to\omega}\tau_n(\varphi(b)) = \lim_{n\to\omega}\tau_n(\psi_n(b)) = \tau(\Psi(b)).
\end{equation}
Hence, $T(\Phi)=T(\Psi)$. Moreover, for every nonzero positive element $b\in B$, the set $\{\tau(\Phi(b)) \mid \tau\in T(A_\infty)\}$ is bounded below by $\inf\{\tau(\varphi(b)) \mid \tau \in T(A)\}$. Since $A$ is simple and unital, and $\varphi$ is an embedding, this bound is strictly positive. So, $\Phi$ (and by \eqref{eqn:omega}, also $\Psi$) maps every trace on $A_\infty$ to a faithful trace on $B$. It follows from \cite[Lemma 4.10 and Proposition 4.17]{Carrion:wz} that $\Phi$ and $\Psi$ are indeed full embeddings.

We are left to contend with $\overline{K}_1^{\mathrm{alg}}$. First suppose that $A$ has real rank zero. Then, as in the proof of Theorem~\ref{thm:class1}, the map $\pi_A \colon \overline{K}_1^{\mathrm{alg}}(A) \to K_1(A)$ in the Thomsen sequence \eqref{eqn:thomsen} is an isomorphism. Let $u\in U_k(B)$ for some $k\in\nn$. Then, for every $n$, $K_1(\psi_n)=K_1(\varphi)$ implies that $[\psi_n(u)^*\varphi(u)] = 0 \in K_1(A)$, hence $[\psi_n(u)^*\varphi(u)]_{\mathrm{alg}} = 0 \in \overline{K}_1^{\mathrm{alg}}(A)$, that is, $\psi_n(u)^*\varphi(u) \in \overline{DU_\infty(A)}$. Since $A$ is a simple, unital, $\js$-stable, tracial $\cs$-algebra, it has stable rank one, meaning that invertible elements are dense in $A$ (see \cite[Theorem 6.7]{Rordam:2004kq}). By \cite[Corollary 3.11]{Gong:2015ab}, we therefore in fact have $\psi_n(u)^*\varphi(u) \in \overline{DU_k(A)}$. Hence, $\Psi(u)^*\Phi(u) = [(\psi_n(u)^*\varphi(u))_{n=1}^\infty] \in \overline{DU_k(A_\infty)}$, which implies that $[\Psi(u)^*\Phi(u)]_{\mathrm{alg}} = 0 \in \overline{K}_1^{\mathrm{alg}}(A_\infty)$. Thus, $\overline{K}_1^{\mathrm{alg}}(\Phi) = \overline{K}_1^{\mathrm{alg}}(\Psi)$.

Finally, suppose that $K_1(B)=0$. In this case, the map $\lambda_B \colon  \mathrm{Aff}(T(B))/\overline{\rho_B(K_0(B))} \to \overline{K}_1^{\mathrm{alg}}(B)$ in the Thomsen sequence \eqref{eqn:thomsen} is an isomorphism. Just as in the proof of Theorem~\ref{thm:class1}\eqref{it:class1}, this automatically implies that $\overline{K}_1^{\mathrm{alg}}(\Phi) = \overline{K}_1^{\mathrm{alg}}(\Psi)$. That said, for the sake of demystifying the Thomsen sequence, let us see how to prove that $\varphi$ and $\psi_n$ \emph{approximately} agree in  $\overline{K}_1^{\mathrm{alg}}$. The map $\lambda_B$ sends the class of an element of $\mathrm{Aff}(T(B))$, which can be represented by the evaluation functional $\widehat b \colon \tau \mapsto \tau(b)$ associated with a self-adjoint element $b\in B$, to the class $[e^{2\pi ib}]_{\mathrm{alg}}$. For any $b\in B_{sa}$ and for every $\eps\in\left(0,\frac{1}{2}\right)$, we have from the uniform continuity of $\widehat b$ that
\[
\|\widehat{\varphi(b)} - \widehat{\psi_n(b)}\| = \sup_{\tau\in T(A)}|\tau\circ\varphi(b) - \tau\circ\psi_n(b)| < \frac{\eps}{2\pi}
\]
for all sufficiently large $n$. In this case, \cite[Proposition 2.1]{Carrion:wz} allows us to find $b_n \in A_{sa}$ such that $\widehat{b_n} = \widehat{\varphi(b)} - \widehat{\psi_n(b)}$ and $2\pi\|b_n\| < \eps$. Since $b_n$ and $\varphi(b)-\psi_n(b)$ represent the same element of $\mathrm{Aff}(T(A))$, we have
\[
e^{2\pi ib_n} =e^{2\pi i(\varphi(b)-\psi_n(b))} \:\text{modulo}\: \overline{DU(A)}.
\]
See the discussion before \cite[Proposition 2.10]{Carrion:wz}, where it is also noted that
\[
e^{2\pi i(\varphi(b)-\psi_n(b))} = e^{2\pi i\varphi(b)}e^{-2\pi i\psi_n(b)} \:\text{modulo}\: \overline{DU(A)}.
\]
Hence, there is a unitary $v_n \in \overline{DU(A)}$ such that
\[
e^{2\pi i\varphi(b)}e^{-2\pi i\psi_n(b)} = e^{2\pi ib_n}v_n.
\]
Since $2\pi\|b_n\| < \eps < \frac{1}{2}$, we have
\[
\|e^{2\pi ib_n}-1\| \le |e^{\|2\pi ib_n\|}-1| < 2\pi\|b_n\| < \eps,
\]
so
\[
d(e^{2\pi i\varphi(b)}e^{-2\pi i\psi_n(b)},\overline{DU(A)}) \le \|e^{2\pi i\varphi(b)}e^{-2\pi i\psi_n(b)} -v_n \| = \|e^{2\pi ib_n}v_n - v_n\| = \|e^{2\pi ib_n}-1\| < \eps
\]
for all large $n$. This shows that $\Phi(e^{2\pi i b})\Psi(e^{-2\pi i b}) \in \overline{DU(A_\infty)}$, and we once again conclude that $\overline{K}_1^{\mathrm{alg}}(\Phi) = \overline{K}_1^{\mathrm{alg}}(\Psi)$. This completes the proof.
\end{proof}

\section{Unitary orbits in $\mathrm{II}_1$ factors} \label{section:II1}

We begin with the observation that the distances $d_{U,\fp}$, $\fp\in[1,\infty]$, are all metrics on the set of \emph{norm}-closed unitary orbits of normal elements in any $\mathrm{II}_1$ factor.

\begin{theorem}[Ding--Hadwin] \label{thm:hadwin}
Let $A$ be a unital commutative $\cs$-algebra, and let $\cM$ be a $\mathrm{II}_1$ factor with (unique) trace $\tau=\tau_\cM$. The following are equivalent for unital $^*$-homomorphisms $\varphi,\psi \colon A \to \cM$:
\begin{enumerate}
\item \label{it:hadwin1} $\varphi$ and $\psi$ are approximately unitarily equivalent in norm;
\item \label{it:hadwin2} there is a sequence $(u_n)_{n=1}^\infty$ of unitaries in $\cM$ such that, for every $\fp\in[1,\infty]$ and every $a\in A$, $\lim_{n\to\infty}\|\varphi(a)-u_n\psi(a)u_n^*\|_\fp =0$;
\item \label{it:hadwin3} there is a sequence $(u_n)_{n=1}^\infty$ of unitaries in $\cM$ such that, for some $\fp\in[1,\infty]$ and every $a\in A$, $\lim_{n\to\infty}\|\varphi(a)-u_n\psi(a)u_n^*\|_\fp =0$;
\item \label{it:hadwin4} $\tau_\cM\circ\varphi = \tau_\cM\circ\psi$.
\end{enumerate}
\end{theorem}

\begin{proof}
\eqref{it:hadwin1}$\implies$\eqref{it:hadwin2} holds because $\|\cdot\|_\fp\le \|\cdot\|_\infty = \|\cdot\|$ for every $\fp$ (see Proposition~\ref{prop:metric}\eqref{it:metric2}). \eqref{it:hadwin2}$\implies$\eqref{it:hadwin3} is trivial. For \eqref{it:hadwin3}$\implies$\eqref{it:hadwin4}, we use the well-known fact (cf.\ \cite[Section 3]{Nelson:1974aa}) that
\begin{equation} \label{eqn:cs}
|\tau(xy)| \le \|x\|_1\|y\|
\end{equation}
for every $x,y\in\cM$. This follows from the polar decomposition $x=u|x|=u|x|^{\frac{1}{2}}|x|^{\frac{1}{2}}$ and the Cauchy--Schwarz inequality:
\[
|\tau(xy)| \le \tau\left(\left(u|x|^{\frac{1}{2}}\right)^*\left(u|x|^{\frac{1}{2}}\right)\right)^{\frac{1}{2}} \tau\left(|x|^{\frac{1}{2}}yy^*|x|^{\frac{1}{2}}\right)^{\frac{1}{2}} \le \tau(|x|)^{\frac{1}{2}} \tau(|x|)^{\frac{1}{2}} \|y\| = \tau(|x|)\|y\| = \|x\|_1\|y\|.
\]
By monotonicity (Proposition~\ref{prop:metric}\eqref{it:metric2}), we then have for every $n\in\nn$ and $a\in A$ that
\begin{align*}
|\tau_\cM(\varphi(a)-\psi(a))| &= |\tau_\cM(\varphi(a)-u_n\psi(a)u_n^*)|\\ &\le  \|\varphi(a)-u_n\psi(a)u_n^*\|_1 \\ &\le \|\varphi(a)-u_n\psi(a)u_n^*\|_\fp
\end{align*}
where $\fp$ and $(u_n)_{n=1}^\infty$ are as in \eqref{it:hadwin3}. It follows that $\tau_\cM(\varphi(a)) = \tau_\cM(\psi(a))$ for every $a\in A$, that is, \eqref{it:hadwin4} holds. The final implication \eqref{it:hadwin4}$\implies$\eqref{it:hadwin1} is \cite[Theorem 3]{Ding:2005aa}.
\end{proof}

\begin{corollary} \label{cor:metric}
Let $\cM$ be a $\mathrm{II}_1$ factor with trace $\tau=\tau_\cM$, and let $\fp\in[1,\infty]$. Then, $d_{U,\fp}$ defines a metric on the set of norm closures of unitary orbits of normal elements in $\cM$.
\end{corollary}

\begin{proof}
Let $x,y\in \cM$ be normal with $d_{U,\fp}(x,y)=0$. We must show that $x$ is approximately unitarily equivalent to $y$ in the norm topology. Note that $d_{U,1}(x,y)=0$. Together with \eqref{eqn:cs}, this implies that $\tau(x^m) = \tau(y^m)$ for every $m\in\nn$, hence $\mu_{\tau,x} = \mu_{\tau,y}$, hence $\sigma(x) = \supp(\mu_{\tau,x}) = \supp(\mu_{\tau,y}) = \sigma(y)$. Theorem~\ref{thm:hadwin} then implies that the functional calculus maps $\varphi_x,\varphi_y\colon C(\sigma(x)) \to \cM$ are approximately unitarily equivalent, which concludes the proof.
\end{proof}

\begin{remark} \label{rem:topology}
If $\fp$ is finite, then $d_{U,\fp}$ is topologically distinct from $d_U=d_{U,\infty}$. Note in particular that in a $\mathrm{II}_1$ factor, $d_{U,2}$ corresponds to the weak$^*$-topology
on spectral measures (by Corollary~\ref{cor:hw} and Proposition~\ref{prop:wassfacts}\eqref{it:wassfacts1}) while $d_U$ corresponds to the much stronger $W_\infty$-topology (by \cite[Theorems 1.1 and 1.4]{Hiai:1989aa}).
\end{remark}

\begin{remark} \label{rem:hadwin}
Though we do not need it in this article, we note that Theorem~\ref{thm:hadwin} holds for a much wider class of domains $A$, including separable, unital, approximately subhomogeneous $\cs$-algebras \cite{Li:2023aa}, and separable, unital, nuclear $\cs$-algebras satisfying the universal coefficient theorem \cite{Hua:2026aa}, further assuming in the latter case that $\varphi,\psi \colon A \to \cM$ are injective.
\end{remark}

\begin{lemma} \label{lemma:continuity}
Let $A$ be a unital $\cs$-algebra with a faithful trace. Then, for any $\fp\in[1,\infty]$, the functions $d_{U,\fp}$ and $W_\fp$ are norm continuous in each variable. 
\end{lemma}

\begin{proof}
Continuity of $d_{U,\fp}$ holds because $d_{U,\fp}(x,y) \le d_U(x,y) \le \|x-y\|$. Continuity of $W_\fp\le W_\infty$ follows from the result for $W_\infty$, which is proved in \cite[Lemma 3.3(i)]{Jacelon:2021wa} (see also \cite[Corollary 3.5]{Jacelon:2021wa}) using \cite[Theorem 1.1]{Hiai:1989aa}.
\end{proof}

\begin{theorem}[Hiai--Nakamura] \label{thm:hn}
Let $\cM$ be a $\mathrm{II}_1$ factor with (unique) trace $\tau_\cM$, and let $\fp\in[1,\infty]$. Then, $d_{U,\fp}(x,y) \le W_\fp(x,y)$ for every pair of normal elements $x,y\in\cM$.
\end{theorem}

\begin{proof}
We adapt the proof of \cite[Theorem 1.4]{Hiai:1989aa}, which establishes the inequality for $p=\infty$ and $\cM$ a $\sigma$-finite semifinite factor. We assume here that $\fp\in[1,\infty)$. By the Borel functional calculus and Lemma~\ref{lemma:continuity}, we may assume that $x$ and $y$ have finite spectra. More precisely, we will assume that there are complex numbers $\alpha_1,\dots,\alpha_m,\beta_1,\dots,\beta_n$ and sets of mutually orthogonal projections $\{p_1,\dots,p_m\}$ and $\{q_1,\dots,q_n\}$ such that $\sum_{i=1}^mp_i = \sum_{j=1}^nq_j =1$, $x=\sum_{i=1}^m\alpha_ip_i$ and $y=\sum_{j=1}^n\beta_jq_j$. The measures induced by $\tau_\cM$ on $\sigma(x)=\{\alpha_1,\dots,\alpha_m\}$ and $\sigma(y)=\{\beta_1,\dots,\beta_n\}$ are $\mu=\sum_{i=1}^m\tau_\cM(p_i)\delta_{\alpha_i}$ and $\nu=\sum_{j=1}^n\tau_\cM(q_j)\delta_{\beta_j}$, respectively. Let $\pi$ be a coupling between $\mu$ and $\nu$ such that the infimum defining the $\fp$-Wasserstein distance $W_\fp(\mu,\nu)$ is attained at $\pi$. As described in Example~\ref{ex:finite}, $\pi$ can be represented as an $m\times n$ matrix $(a_{ij})$. Since $\cM$ is a $\mathrm{II}_1$ factor, we can find sets of orthogonal projections $\{p_{ij}\}$ and $\{q_{ij}\}$, $(i,j) \in \{1,\dots,m\}\times\{1,\dots,n\}$, such that $p_i=\sum_{j=1}^np_{ij}$, $q_j=\sum_{i=1}^mq_{ij}$ and $\tau_\cM(p_{ij}) = a_{ij} = \tau_\cM(q_{ij})$ for every $i$ and $j$. Since each $p_{ij}$ is unitarily equivalent to $q_{ij}$ (by virtue of having the same trace), and the different $p_{ij}$ and $q_{ij}$ are orthogonal, we can find $u\in U(\cM)$ such that $uyu^*= \sum_{i,j}\beta_jp_{ij}$. For any $\fp\in[1,\infty)$, we then have
\[
d_{U,\fp}(x,y) \le \|x-uyu^*\|_\fp = \left\|\sum_{ij}\alpha_ip_{ij} - \sum_{i,j}\beta_jp_{ij}\right\|_\fp = \left(\sum_{ij}a_{ij}|\alpha_i-\beta_j|^\fp\right)^{\frac{1}{\fp}} = W_\fp(x,y). \qedhere
\]
\end{proof}

As pointed out by the referee, our proof of the reverse inequality (Theorem~\ref{thm:hw}) will be familiar to those working in quantum information theory. Moreover, the result can be deduced from \cite[Theorem 1.5]{Biane:2001aa} once one interprets a normal operator as a pair of commuting self-adjoint operators. The argument below, as well as Biane and Voiculescu's, amounts to an application of the Namioka--Connes trick described in \cite[Proposition 10.3.3]{Anantharaman-Delaroche:aa}.

\begin{theorem}[Hoffman--Wielandt; Biane--Voiculescu] \label{thm:hw}
Let $\cM$ be a von Neumann algebra equipped with a faithful trace $\tau$. Then, $W_{2,\tau}(x,y) \le \|x-y\|_{2,\tau}$ for any pair of normal elements $x,y\in\cM$.
\end{theorem}

\begin{proof}
As in the proof of Theorem~\ref{thm:hn}, we may assume that there are sets of mutually orthogonal projections $\{p_1,\dots,p_m\}$ and $\{q_1,\dots,q_n\}$ such that $\sum_{i=1}^mp_i = \sum_{j=1}^nq_j =1$, $x=\sum_{i=1}^m\alpha_ip_i$ and $y=\sum_{j=1}^n\beta_jq_j$. The $m\times n$ matrix $(a_{ij}:=\tau(p_iq_j))$ is a coupling between $\mu_{\tau,x}=\sum_{i=1}^m\tau(p_i)\delta_{\alpha_i}$ and $\mu_{\tau,y}=\sum_{j=1}^n\tau(q_j)\delta_{\beta_j}$. By definition of $W_2$, we have
\begin{align*}
\|x-y\|_{2,\tau}^2 &= \tau\left(\left(\sum_{i=1}^m\alpha_ip_i - \sum_{j=1}^n\beta_jq_j\right)^*\left(\sum_{i=1}^m\alpha_ip_i - \sum_{j=1}^n\beta_jq_j\right)\right)\\
&= \sum_i|\alpha_i|^2\tau(p_i) - \sum_{i,j}\overline{\alpha_i}\beta_j\tau(p_iq_j) - \sum_{i,j}\overline{\beta_i}\alpha_j\tau(q_jp_i) + \sum_j|\beta_j|^2\tau(q_j)\\
&= \sum_{i,j}a_{ij}|\alpha_i-\beta_j|^2\\
&\ge \inf_{\pi\in\Pi(\mu_{\tau,x},\mu_{\tau,y})} \int_{\sigma(x)\times \sigma(y)} |w-z|^2 \,d\pi(w,z) = W_{2,\tau}(x,y)^2. \qedhere
\end{align*}
\end{proof}

\begin{corollary} \label{cor:hw}
If $\cM$ is a $\mathrm{II}_1$ factor and $x,y\in\cM$ are normal, then $d_{U,2}(x,y)=W_2(x,y)$.
\end{corollary}

\begin{proof}
Since the distance $W_2(x,y)$ is invariant under unitary conjugation, Theorem~\ref{thm:hw} implies that $W_2(x,y) \le d_{U,2}(x,y)$. By Theorem~\ref{thm:hn}, equality holds.
\end{proof}

Recall that a \emph{length space} is a metric space for which the distance between two points is equal to the infimum of the lengths of paths between them (see \cite[Chapter 2]{Burago:2001aa}). Every complete, locally compact length space is \emph{strictly intrinsic}, meaning that there is always a length-minimising geodesic witnessing this infimum (see \cite[Theorem 2.5.23]{Burago:2001aa}).

\begin{theorem} \label{thm:lengthspaces}
Let $\cM$ be a $\mathrm{II}_1$ factor, and let $\frak{n}(\cM)$ and $\frak{sa}(\cM)$ denote the sets of approximate unitary equivalence classes of contractive normal and contractive self-adjoint elements of $\cM$, respectively.
\begin{enumerate}[(1)]
\item \label{it:length1} The metric space $(\frak{n}(\cM),d_{U,2})$ is isometric to the $2$-Wasserstein space $(\prob(\mathbb{D}),W_2)$ of probability measures on the unit disc. In particular, it is a compact length space.
\item \label{it:length2} For every $\fp\in[1,\infty)$, the metric space$(\frak{sa}(\cM),d_{U,\fp})$ is isometric to the compact length space $(\prob([-1,1]),W_\fp)$.
\item \label{it:length3} The metric space $(\frak{sa}(\cM),d_{U})$ is isometric to the (noncompact) complete, strictly intrinsic length space $(\prob([-1,1]),W_\infty)$.
\end{enumerate}
\end{theorem}

\begin{proof}
Recall from Corollary~\ref{cor:metric} that the distances $d_{U,\fp}$ are indeed metrics on $\frak{n}(\cM)\supseteq\frak{sa}(\cM)$.
\begin{enumerate}[(1)]
\item The map $(\frak{n}(\cM),d_{U,2}) \to (\prob(\sigma(x)),W_2) \subseteq (\prob(\mathbb{D}),W_2)$, $[x] \mapsto \mu_{\tau_\cM,x}$ is surjective. To see this, let $\lambda$ be the Lebesgue measure on $[0,1]$ and let $\varphi \colon L^\infty([0,1],\lambda) \to \cM$ be a unital embedding such that $\tau_\cM \circ \varphi = \tau_\lambda := \int_0^1 \cdot \, d\lambda$. For example, we could embed $L^\infty([0,1],\lambda)$ as a maximal abelian subalgebra of the hyperfinite $\mathrm{II}_1$ factor $\mathcal{R}$ (see \cite[Theorem 3.2.4 and Exercise 3.3]{Anantharaman-Delaroche:aa}), which embeds as a subfactor of $\cM$ (see \cite[Proposition 4.2.6]{Anantharaman-Delaroche:aa}, or indeed \cite[Theorem XIII]{Murray:1943le}). Given $\mu\in\prob(\mathbb{D})$, we can find a Borel function $f\colon[0,1]\to\mathbb{D}$ such that $f_*\lambda=\mu$ (see, for example, \cite[Theorem 8.5.4]{{Bogachev:2007aa}}). Then, $\varphi(f)\in\frak{n}(\cM)$ satisfies $\mu_{\tau_\cM,\varphi(f)}=\mu$, so the function $[x] \mapsto \mu_{\tau_\cM,x}$ is indeed surjective. By Corollary~\ref{cor:hw}, it is also isometric. Since $\mathbb{D}\subseteq\cc$ is compact and convex, it is a compact length space (length-minimising geodesics are straight lines). It follows from \cite[Corollary 2.7]{Lott:2009aa} that $(\prob(\mathbb{D}),W_2) \cong (\frak{n}(\cM),d_{U,2})$ is also a compact length space. Moreover, by \cite[Proposition 2.10]{Lott:2009aa}, every length-minimising geodesic between classes $[x_0],[x_1]\in\frak{n}(\cM)$ is a `displacement interpolation' $\{\mu_t\}_{t\in[0,1]}$ corresponding to an optimal dynamical transference plan that describes how the spectral measure of $x_0$ is transported over time to the spectral measure of $x_1$.
\item It follows from \cite[Theorem 4.3]{Hiai:1989aa} that the map  $[x]\mapsto\mu_{\tau_\cM,x}$ defines an isometry between $(\frak{sa}(\cM),d_{U,\fp})$ and $(\prob([-1,1]),W_\fp)$, which is a compact length space by \cite[Remark 2.8]{Lott:2009aa}.
\item Similarly, by \cite[Theorem 2.1(3)]{Hiai:1989aa}, $[x]\mapsto\mu_{\tau_\cM,x}$ is an isometry between $(\frak{sa}(\cM),d_{U})$ and $(\prob([-1,1]),W_\infty)$, which is complete and strictly intrinsic by \cite[Proposition 3.4]{Jacelon:2025ab}. \qedhere
\end{enumerate}
\end{proof}

\section{Unitary orbits in classifiable $\mathrm{C}^*$-algebras} \label{section:cstar}

In this section, we use classification \cite{Carrion:wz} and optimal transport theory \cite{Jacelon:2021wa} to move from the von Neumann world to the world of tracial $\cs$-algebras. Before embarking on this task, and with apologies to my coauthors, I would like to comment on an aspect of \cite{Jacelon:2021wa} that I now question. One fact we relied on was \cite[Theorem 2.1(1)]{Hiai:1989aa}, which asserts that $d_U(x,y)=W_\infty(x,y)$ for every pair of commuting normal elements in a semifinite von Neumann algebra. But this cannot be correct as stated, given the counterexamples of \cite{Holbrook:1992aa}, and indeed Hiai and Nakamura's appeal to \cite[Proposition 2.3]{Davidson:1986aa} only seems to give that $W_\infty(x,y) \le \|x-y\|$ for such $x$ and $y$. The effect on \cite{Jacelon:2021wa} is that $d_U(x,y)$ is not the right upper bound for $W_\infty(x,y)$ in \cite[Theorem 4.10]{Jacelon:2021wa}, which instead shows for the normal elements it considers that
\[
d_U(x,y) \le W_\infty(x,y) \le \inf\{\|x-uyu^*\| \mid u\in U(A) \text{ with } [x,uyu^*]=0 \}.
\]
That said, this is the only significant effect, and the \emph{de facto} r\^ole of \cite[Theorem 4.10]{Jacelon:2021wa} is to introduce the idea of continuous transport that is more carefully employed in \cite[Theorem 4.11]{Jacelon:2021wa}. Our other applications of \cite[Theorem 2.1(1)]{Hiai:1989aa} were unnecessary, a fact that I emphasise below in the proof of Theorem~\ref{thm:lowerbound}.

\begin{definition} \label{def:core}
Let $(X,\rho)$ be a compact metric space, let $A$ be a unital $\cs$-algebra with a faithful trace, let $\fp\in[1,\infty]$, and let $\ff=\rr$ or $\ff=\cc$. We write $\mu_\sigma$ for the Borel probability measure on $X$ corresponding to a given trace $\sigma\in T(C(X)) \cong \prob(X)$, and we define
\begin{equation} \label{eqn:lipf}
\mathrm{Lip}^\ff_1(X,\rho) = \{f\colon X\to \ff \mid (\forall\, x,y\in X)\: |f(x)-f(y)|\le \rho(x,y)\}.
\end{equation}
For unital $^*$-monomorphisms $\varphi,\psi\colon C(X)\to A$, we define the \emph{(uniform) $\fp$-norm unitary orbit distance}
\begin{equation} \label{eqn:duf}
d_{U,\fp}^\ff(\varphi,\psi) = \inf_{u\in U(A)} \sup_{f\in\mathrm{Lip}^\ff_1(X,\rho)}\|\varphi(f)-u\psi(f)u^*\|_\fp
\end{equation}
and the \emph{(uniform) tracial $\fp$-Wasserstein distance}
\begin{equation} \label{eqn:tracialwass}
W_\fp(\varphi,\psi) = \sup_{\tau\in T(A)} W_\fp(\mu_{\tau\circ\varphi},\mu_{\tau\circ\psi}).
\end{equation}
\end{definition}

A useful basic fact about equation \eqref{eqn:duf} is that it can equivalently be written as
\begin{equation} \label{eqn:aa}
d_{U,\fp}^\ff(\varphi,\psi) = \inf_{u\in U(A)} \sup_{f\in\mathrm{Lip}^\ff_1(X,\rho)\cap B_{\mathrm{diam}(X,\rho)}}\|\varphi(f)-u\psi(f)u^*\|_\fp
\end{equation}
where, given $r>0$, $B_r=B_r(C(X))$ denotes the ball $\{f\in C(X) \mid \|f\|\le r\}$. This is because $\mathrm{Lip}^\ff_1(X,\rho) = \ff \cdot 1 + \mathrm{Lip}^\ff_1(X,\rho)\cap B_{\mathrm{diam}(X,\rho)}$, so if $\varphi,\psi \colon C(X) \to A$ are \emph{unital} $^*$-homomorphisms, then one reference set is as good as the other for measuring unitary conjugacy.

The distances $d_{U,\fp}^\ff$ are particularly meaningful in the context of classification \cite{Carrion:wz}. If $A$ and $B$ are unital $\cs$-algebras, let us write  $\mathrm{Emb}^1(B,A)$ for the set of unital $^*$-monomorphisms $B\to A$. Recall also the notion of approximate unitary equivalence of $^*$-homomorphisms described in Definition~\ref{def:auequiv}.

\begin{theorem} \label{thm:cstarmetric}
Let $A$ be a simple, separable, unital, nuclear, $\js$-stable $\cs$-algebra with $T(A)\ne\emptyset$, let $(X,\rho)$ be a compact metric space such that $K_*(C(X))$ is finitely generated and torsion free, and let $\varphi\in\mathrm{Emb}^1(C(X),A)$.  Suppose further either that $A$ has real rank zero or that $K_1(C(X))=0$. Then, for every $\fp\in[1,\infty]$ and for $\ff=\rr$ or $\ff=\cc$, $d_{U,\fp}^\ff$ is a metric on the set of (norm) approximate unitary equivalence classes of elements in
\begin{equation} \label{eqn:emb}
\mathrm{Emb}^1(C(X),A)_{\varphi} := \{\psi \in\mathrm{Emb}^1(C(X),A) \mid K_*(\psi) = K_*(\varphi)\}.
\end{equation}
\end{theorem}

\begin{proof}
Let $\fp\in[1,\infty]$. By Proposition~\ref{prop:metric}\eqref{it:metric3}, $d_{U,\fp}^\ff$ is symmetric: $d_{U,\fp}^\ff(\varphi,\psi) = d_{U,\fp}^\ff(\psi,\varphi)$. The triangle inequality is proved in exactly the same way as Proposition~\ref{prop:metric}\eqref{it:metric4}. To see that $d_{U,\fp}^\ff$ is well defined and positive definite on $\fp$-norm approximate unitary equivalence classes of unital embeddings, we appeal to the Arzel\`{a}--Ascoli theorem to deduce that $\mathrm{Lip}^\ff_1(X,\rho) \cap B_{\mathrm{diam}(X,\rho)}$ is compact (in particular, totally bounded). It then follows from \eqref{eqn:aa} that, if $\varphi,\psi \colon C(X) \to A$ are unital embeddings that are approximately unitarily equivalent in $\fp$-norm, then $d_{U,\fp}^\ff(\varphi,\psi) = 0$. Conversely, $d_{U,\fp}^\ff(\varphi,\psi) = 0$ implies that $\varphi$ and $\psi$ are approximately equivalent in $\fp$-norm because the linear span of $\mathrm{Lip}^\ff_1(X,\rho)$ is dense in $C(X)$.

To complete the proof, we show that for every $\psi \in \mathrm{Emb}^1(C(X),A)_{\varphi}$,  $d_{U,\fp}^\ff(\varphi,\psi)=0$ if and only if $\psi$ is approximately unitarily equivalent to $\varphi$ \emph{in norm}. Suppose first that $\psi$ is approximately unitarily equivalent to $\varphi$. Then, $d_{U,\infty}^\ff(\varphi,\psi) = 0$ (as we argued above), so  $d_{U,\fp}^\ff(\varphi,\psi) = 0$ by monotonicity (Proposition~\ref{prop:metric}\eqref{it:metric2}). Suppose conversely that $d_{U,\fp}^\ff(\varphi,\psi)=0$. By definition of $d_{U,\fp}^\ff$, we know that for every $f\in\mathrm{Lip}^\ff_1(X,\rho)$, $\varphi(f)$ is approximately unitarily equivalent in $\fp$-norm to $\psi(f)$, which, as in the proof of Theorem~\ref{thm:hadwin}, implies that $\tau(\varphi(f)) = \tau(\psi(f))$ for every $\tau\in T(A)$. (Note that the polar decomposition in the proof can be taken in the weak closure of the GNS representation associated with $\tau$.) Since the linear span of $\mathrm{Lip}^\ff_1(X,\rho)$ is dense in $C(X)$, it follows that $\tau\circ\varphi=\tau\circ\psi$ for every $\tau\in T(A)$, that is, $T(\varphi)=T(\psi)$. By assumption, we also have $K_*(\psi) = K_*(\varphi)$. Therefore, by Theorem~\ref{thm:class1}\eqref{it:class1}, $\varphi$ and $\psi$ are indeed approximately unitarily equivalent.
\end{proof}
 
For measure-theoretic computation of these metrics, we turn to optimal transport theory.

\begin{definition}[cf.\ {\cite[Definition 2.6]{Jacelon:2021vc}}] \label{def:transport}
For $\fp\in[1,\infty]$, the \emph{continuous $\fp$-transport constant} of a nonempty compact, connected metric space $(X,\rho)$ is the least $k=k_{(X,\rho,\fp)} \in [1,\infty]$ such that, for every pair of faithful, diffuse measures $\mu,\nu\in\prob(X)$ and every $\eps>0$, there is a homeomorphism $h\colon X \to X$ that is homotopic to the identity map and satisfies
\begin{equation} \label{eqn:transport1}
W_\infty(h_*\nu,\mu) \le \eps
\end{equation}
and
\begin{equation} \label{eqn:transport2}
\|\rho(x,h(x))\|_{L^{\fp}(X,\nu)} \le k \cdot W_\fp(\mu,\nu) + \eps.
\end{equation}
\end{definition}

As noted in \cite[Remark 2.7]{Jacelon:2021vc}, all of the Wasserstein metrics are topologically equivalent on the set $\mathcal{P}_f(X)$ of faithful measures on $X$. In other words, for any $\fp'\in[1,\infty]$, any $\mu\in\mathcal{P}_f(X)$ and any $\eps>0$, there exists $\eps'>0$ (very much depending on $\mu$) such that
\[
B_{\mu,\eps'}(\mathcal{P}_f(X),W_{\fp'}) \subseteq B_{\mu,\eps}(\mathcal{P}_f(X),W_\infty) \subseteq B_{\mu,\eps}(\mathcal{P}_f(X),W_{\fp'})
\]
where $B_{y,r}(Y,d)$ denotes the closed ball of radius $r$ centred on $y$ in a metric space $(Y,d)$. The latter of these inclusions is guaranteed by Proposition~\ref{prop:wassfacts}\eqref{it:wassfacts3}. Noting that both $\mu$ and $h_*\nu$ appearing in \eqref{eqn:transport1} are faithful, it follows that if Definition~\ref{def:transport} were to be restated with some other $W_\fp'$ in place of $W_\infty$  in \eqref{eqn:transport1}, then the value of the constant $k_{(X,\rho,\fp)}$ would be the same.

We also automatically have that $k_{(X,\rho,\fp)} \ge 1$. This is because the transport map $h$ defines a transference plan $\pi:=(\id\times h)_*\nu$ between $\nu$ and $h_*\nu$. By definition of $W_\fp$ \eqref{eqn:transfer}, we therefore have
\[
\|\rho(x,h(x))\|_{L^{\fp}(X,\nu)} = \left(\int_{X} \rho(x,h(x))^\fp\, d\nu(x)\right)^\frac{1}{\fp} = \left(\int_{X\times X} \rho(w,z)^\fp\, d\pi(w,z)\right)^\frac{1}{\fp} \ge W_\fp(\nu,h_*\nu)
\]
and so
\[
W_\fp(\mu,\nu) \le W_\fp(\mu,h_*\nu) + W_\fp(\nu,h_*\nu) \le \eps + \|\rho(x,h(x))\|_{L^{\fp}(X,\nu)}.
\]
We now consider cases where this minimal possible value of $k_{(X,\rho,\fp)}$ is actually attained. We look for examples among strictly intrinsic length spaces.

\begin{proposition} \label{prop:interval}
Let $(X,\rho)$ be a Euclidean interval. Then, $k_{(X,\rho,\fp)}=1$ for every $\fp\in[1,\infty]$.
\end{proposition}

\begin{proof}
We may assume that $X=[0,1]$. Let $\mu,\nu\in\prob([0,1])$, and let $F(t)=\mu([0,t))$ and $G(t)=\nu([0,t))$ be their cumulative distribution functions. As noted in \cite[Proposition 2.2]{Jacelon:2021wa} and \cite[Proposition 2.5]{Jacelon:2021vc}, if $\mu$ and $\nu$ are faithful and diffuse, then the increasing rearrangement map $h=F^{-1}\circ G \colon [0,1] \to [0,1]$ is a homeomorphism that satisfies $h_*\nu=\mu$ and $\|\rho(x,h(x))\|_{L^{\fp}([0,1],\nu)} \le W_\fp(\mu,\nu)$ for every $\fp$. By definition, this demonstrates that $k_{(X,\rho,\fp)}=1$.
\end{proof}

\begin{proposition} \label{prop:circle}
Let $(X,\rho)$ be a circle with its intrinsic metric. Then, $k_{(X,\rho,\fp)}=1$ for every $\fp\in[1,\infty]$.
\end{proposition}

\begin{proof}
Identify $X=\mathbb{T}$ with $\rr/2\pi\zz$, and let $\mu,\nu\in\prob(X)$ be faithful and diffuse. We claim that, just as for intervals, the optimal transference plan between $\mu$ and $\nu$ is witnessed by a transport homeomorphism $h\colon X\to X$. For finite $\fp$, this follows from \cite{Delon:2010aa}, which studies optimal transport on the circle for cost functions $\widehat c \colon \mathbb{T} \times \mathbb{T} \to \mathbb{R}$ that correspond to periodic functions $c \colon \rr\times\rr\to\rr$
via
\[
\widehat c(w,z) := \inf\{c(x,y) \mid e^{i x}=w, e^{i y}=z\}.
\]
The cost functions are assumed to satisfy the Monge condition
\begin{equation} \label{eqn:monge}
c(x_1,y_1) + c(x_2,y_2) < c(x_1,y_2) + c(x_2,y_1)
\end{equation}
whenever $x_1<x_2$ and $y_1<y_2$. As noted in \cite{Delon:2010aa}, for $\fp\in(1,\infty)$ this applies to the function $c(x,y)=|x-y|^\fp$, which corresponds to the cost $\widehat c(w,z) = \rho(w,z)^\fp$ appearing in the definition \eqref{eqn:transfer} of the $\fp$-Wasserstein distance. The optimal transference plans between $\mu$ and $\nu$ must be `locally optimal' in the sense introduced in \cite{Delon:2010aa}, roughly meaning that the cost of the plan cannot be reduced by local perturbation. The Monge condition forces locally optimal transport plans to be monotone. In fact, by \cite[Theorem 8]{Delon:2010aa}, a locally optimal plan must be a shift, meaning that the circle can be cut (at one location for $\mu$ and a shifted location for $\nu$) so that, just as in Proposition~\ref{prop:interval}, the transport map is the increasing rearrangement homeomorphism associated with the cumulative distribution functions determined by these cuts. The `globally optimal' shifts, that is, those associated with optimally located cuts, are exactly the cost-minimising transport plans (see \cite[Theorem 20]{Delon:2010aa}).  These optimal shifts demonstrate that $k_{(X,\rho,\fp)}=1$ for $\fp\in(1,\infty)$. Note in particular that they have winding number $1$, so are indeed homotopic to the identity map on the circle. The $\fp=1$ problem is addressed in \cite{Cabrelli:1995aa}, but as remarked in \cite{Delon:2010aa}, it can also be deduced as a limiting case. More precisely, if $\fp=1$ or indeed $\fp=\infty$, we let $(\fp_n)_{n=1}^\infty$ be a sequence in $(1,\infty)$ converging to $\fp$ and we decide where to cut the circle by taking an accumulation point of cut points for $\fp_n$. Since $W_\fp(\mu,\nu) = \lim_{n\to\infty}W_{\fp_n}(\mu,\nu)$, it follows that the corresponding shift has the required properties.
\end{proof}

\begin{theorem}[cf.\ {\cite[Proposition 2.5]{Jacelon:2021wa}}] \label{thm:infinitycircle}
Let $X\subseteq\cc$ be a planar circle, that is, a circle equipped with the ambient Euclidean metric $\rho$. Then, $k_{(X,\rho,\infty)}=1$ and $k_{(X,\rho,\fp)} \le \frac{\pi}{2}$ for every $\fp\in[1,\infty)$.
\end{theorem}

\begin{proof}
The Euclidean metric $\rho$ is a strictly increasing continuous function of the geodesic distance $\gamma$. Indeed, if $r$ denotes the radius of the circle, then for every $x,y\in X$ we have $\gamma(x,y)\in[0,\pi r]$ and $\rho(x,y) = 2r\sin\frac{\gamma(x,y)}{2r}=:f(\gamma(x,y))$. By Proposition~\ref{prop:circle}, given faithful, diffuse measures $\mu,\nu\in\prob(X)$ there is a homeomorphism $h\colon X \to X$ such that $h_*\nu=\mu$ and $\sup_{x\in X}\gamma(x,h(x))=W_{\infty,\gamma}(\mu,\nu)$. Note that
\begin{align} \label{eqn:winftychange}
W_{\infty,\gamma}(\mu,\nu) &= \inf\{t>0 \mid(\forall\, U\subseteq X\text{ Borel}) \:\mu(U) \le \nu(\{x \in X \mid \gamma(x,U)<t\})\} \notag \\
&= \inf\{t>0 \mid(\forall\, U\subseteq X\text{ Borel}) \:\mu(U) \le \nu(\{x \in X \mid \rho(x,U)<f(t)\})\} \notag \\
&= f^{-1}(W_{\infty,\rho}(\mu,\nu)).
\end{align}
We conclude from this that
\[
\sup_{x\in X}\rho(x,h(x)) = f\left(\sup_{x\in X}\gamma(x,h(x))\right) = f(W_{\infty,\gamma}(\mu,\nu)) = W_{\infty,\rho}(\mu,\nu)
\]
and hence that $k_{(X,\rho,\infty)}=1$. Equation \eqref{eqn:winftychange} is peculiar to $\fp=\infty$, so we cannot deduce in the same way that $k_{(X,\rho,\fp)}=1$ for finite $\fp$. The proof of Proposition~\ref{prop:circle} is also not immediately adaptable, because for the cost associated with Euclidean distance, the necessary convexity for establishing the Monge condition \eqref{eqn:monge} only holds on circular arcs no larger than a quarter circle. That said, we can conclude that $k_{(X,\rho,\fp)} \le \frac{\pi}{2}$ because the uniform bound $\rho(x,y) \le \gamma(x,y) \le \frac{\pi}{2}\,\rho(x,y)$ implies that the transport map $h\colon X \to X$ provided by Proposition~\ref{prop:circle} satisfies
\[
\|\rho(x,h(x))\|_{L^{\fp}(X,\nu)} \le \|\gamma(x,h(x))\|_{L^{\fp}(X,\nu)} = W_{\fp,\gamma}(\mu,\nu) \le \tfrac{\pi}{2}\,W_{\fp,\rho}(\mu,\nu). \qedhere
\]
\end{proof}

\begin{theorem} \label{thm:convex}
Let $X$ be a nonempty compact, convex subset of $\rr^n$, for some $n\ge1$, equipped with the ambient Euclidean metric $\rho$. Then, $k_{(X,\rho,\fp)}=1$ for every $\fp\in[1,\infty]$.
\end{theorem}

\begin{proof}
The $n=1$ case of the theorem is Proposition~\ref{prop:interval}, and the $\fp=\infty$ case is \cite[Theorem 2.13]{Jacelon:2021wa}. Let us sketch the proof of this latter result and indicate how it should be modified for other $\fp$. By reducing $n$ if necessary, we may assume that the interior of $X$ is nonempty. Given faithful measures $\mu,\nu\in\prob(X)$, we can find finitely supported measures $\mu'=\frac{1}{n}\sum_{i=1}^n\delta_{x_i}$ and $\nu'=\frac{1}{n}\sum_{i=1}^n\delta_{y_i}$ such that $W_\infty(\mu,\mu')$ and $W_\infty(\nu,\nu')$ are small (see \cite[Lemma 2.3]{Jacelon:2021wa}) and with $(x_1,\dots,x_n)$ and $(y_1,\dots,y_n)$ located in the interior of $X$ and ordered so that $W_\infty(\mu',\nu')=\max_{1\le i\le n}\rho(x_i,y_i)$, or so that $W_\fp(\mu',\nu')^\fp = \frac{1}{n} \sum_{i=1}^n \rho(x_i,y_i)^\fp$ if $\fp<\infty$. We join $x_i$ to $y_i$ by shortest paths (that is, straight lines) $\gamma_i$, which we then perturb to be pairwise disjoint. If $n\ge3$, this perturbation is done by jumping one path over another at any intersection point. We lack the dimensions for this move if $n=2$, but a more delicate procedure can be carried out as in \cite[Proposition 2.9]{Jacelon:2021wa}. We encase these paths in disjoint tubes, within which we find homeomorphisms $h_i$ interchanging $x_i$ and $y_i$ and fixing all points on the boundaries of the tubes. The global transport map $h$ is defined to be the union of the $h_i$ within the tubes and the identity on the complement. As explained in \cite[Theorem 6.4]{Jacelon:2025ab}, the region in which $h$ is not approximately $1$-Lipschitz can be arranged to have very small $\nu$-measure. This guarantees that $W_\infty(h_*\nu,\mu') = W_\infty(h_*\nu,h_*\nu')$ is small, hence so is $W_\infty(h_*\nu,\mu)$, giving us \eqref{eqn:transport1}. The significance of the $\gamma_i$ being shortest paths is to ensure that $\|\rho(h(x),x)\|$ is close to $W_\infty(\mu,\nu)$, but this is not relevant for $\fp<\infty$. Indeed, if $\fp<\infty$, then
\begin{equation} \label{eqn:wpotimal}
\|\rho(h(x),x)\|^\fp_{L^{\fp}(\nu)} = \frac{1}{n} \sum_{i=1}^n\rho(h(y_i),y_i)^\fp = \frac{1}{n} \sum_{i=1}^n\rho(x_i,y_i)^\fp = W_\fp(\mu',\nu')^\fp \approx W_\fp(\mu,\nu)^\fp
\end{equation}
regardless of the geometry of the paths. This gives us \eqref{eqn:transport2} with constant $k=1$.
\end{proof}

\begin{theorem} \label{thm:manifolds}
Let $X$ be a compact, connected Riemannian manifold with intrinsic metric $\rho$. Then, $k_{(X,\rho,\fp)}=1$ for every $\fp\in[1,\infty)$, and $k_{(X,\rho,\infty)}=1$ if $\dim X \ne 2$.
\end{theorem}

\begin{proof}
If $\dim X=1$, then $X$ is a circle and the theorem is covered by Proposition~\ref{prop:circle}. In dimension at least three, the $\fp=\infty$ case of the theorem is \cite[Theorem 2.8]{Jacelon:2021vc} (see also \cite[Theorem 6.4]{Jacelon:2025ab}). The proof is the same as that of Theorem~\ref{thm:convex}, with straight lines replaced by length-minimising geodesics. For finite $\fp$, the argument also works in dimension two, even without the careful analysis of \cite[Proposition 2.9]{Jacelon:2021wa}. Intersecting geodesics can be disentangled one at a time by going \emph{around} rather than over, with the detours confined to regions of small measure. Unlike the situation for $\fp=\infty$, it does not matter that this may result in longer paths. We still obtain \eqref{eqn:wpotimal} and deduce that $k_{(X,\rho,\fp)}=1$.
\end{proof}

\begin{theorem} \label{thm:transport}
Let $(X,\rho)$ be a nonempty compact, connected metric space such that $K^*(X)$ is finitely generated and torsion free, and let $\fp\in[1,\infty]$. Let $A$ be a simple, separable, unital, nuclear, $\js$-stable $\cs$-algebra with $T(A)\ne\emptyset$. Suppose either that $K^1(X)=0$ and $A$ has a unique trace, or that $A$ has real rank zero and finitely many extremal traces. Then,
\[
d^\cc_{U,\fp}(\varphi,\psi) \le k_{(X,\rho,\fp)} \cdot W_\fp(\varphi,\psi)
\]
for every pair of unital $^*$-monomorphisms $\varphi,\psi\colon C(X)\to A$ with $K_*(\varphi)=K_*(\psi)$.
\end{theorem}

\begin{proof}
The $\fp=\infty$ case of the theorem is proved in \cite[Theorem 4.11]{Jacelon:2021wa} and \cite[Theorem 3.1]{Jacelon:2021vc} (the latter result explaining how to relax the $K$-theoretic and tracial restrictions appearing in the former). The proof relies on Theorem~\ref{thm:class2}, which tells us that unital embeddings $C(X)\to A$ that agree on $K$-theory and agree sufficiently closely on traces are unitarily conjugate up to a given $\eps$ on the compact set $\mathrm{Lip}^\cc_1(X,\rho)\cap B_{\mathrm{diam}(X,\rho)}(C(X))$ (retaining the notation used in \eqref{eqn:aa}). Let us briefly summarise the argument and adapt it to the case of finite $\fp$.

Let $\{\tau_1,\dots,\tau_m\}$ be the extremal traces of $A$, and let $\eps>0$. If $A$ has a unique trace, then $m=1$. The first steps of the argument are perturbation (so that all the measures involved are diffuse) and diagonalisation (if $m>1$). More precisely, we may assume via classification that:
\begin{enumerate}[(i)]
\item \label{it:perturb1} for each $i$, $\mu_i:=\mu_{\tau_i\circ\varphi}$ and $\nu_i:=\mu_{\tau_i\circ\psi}$ are (faithful and) diffuse measures on $X$;
\item \label{it:perturb2} there are orthogonal projections $p_1,\dots,p_m \in A$ with $p_1+\dots+p_m=1$, and unital $^*$-homomorphisms $\psi_i\colon C(X) \to p_iAp_i$, such that $\psi=\psi_1+\dots+\psi_m$ and $|\tau_i(p_j)-\delta_{ij}| < \min\{\eps,(2\mathrm{diam}(X,\rho))^{-\fp}\eps\}$  for every $i$ and $j$; in particular,  $|\tau_i\circ\psi(g)-\tau_i\circ\psi_i(g)|<\eps$ for every $i$ and every $g\in C(X)$ with $\|g\|\le (2\mathrm{diam}(X,\rho))^{\fp}$.
\end{enumerate}
In addition to local uniqueness (Theorem~\ref{thm:class2}), the procedure uses the existence part of \cite[Theorem B]{Carrion:wz}, or rather the special cases \cite[Theorem 0.1]{Ng:2008aa} and \cite[Proposition 4.5]{Jacelon:2021wa} (see \cite[Propositions 4.7, 4.9]{Jacelon:2021wa}). Real rank zero is also a crucial assumption for the diagonalisation step \eqref{it:perturb2} (see \cite[Lemma 4.8]{Jacelon:2021wa}). 

Now assume that $k_{(X,\rho,\fp)}<\infty$. By definition (Definition~\ref{def:transport}), this implies that for each $i$, we can find a homeomorphism $h_i\colon X\to X$ such that $(h_i)_*\nu_i$ is close to $\mu_i$ and
\begin{equation} \label{eqn:transport2i}
\|\rho(x,h_i(x))\|_{L^{\fp}(X,\nu_i)}^\fp \le k_{(X,\rho,\fp)}^\fp \cdot W_\fp(\mu_i,\nu_i)^\fp + \eps.
\end{equation}
The map $\varphi':=\sum_{i=1}^m\psi_i\circ h_i^* \colon C(X) \to A$ is a unital embedding that agrees with $\varphi$ on $K$-theory (because the $h_i$ are homotopic to the identity map) and approximately agrees on traces. Theorem~\ref{thm:class2} allows us to conclude that $d^\cc_{U,\infty}(\varphi,\varphi')$ is small enough so that
\begin{equation} \label{eqn:bound3}
|d^\cc_{U,\fp}(\varphi,\psi)^\fp - d^\cc_{U,\fp}(\varphi',\psi)^\fp| < \eps.
\end{equation}
By convexity of $W_\fp^\fp$ (Proposition~\ref{prop:wassfacts}\eqref{it:wassfacts2}), we have
\begin{equation} \label{eqn:bound2}
W_\fp(\varphi,\psi)^\fp = \sup_{\tau\in T(A)} W_\fp(\mu_{\tau\circ\varphi},\mu_{\tau\circ\psi})^\fp = \max_{1\le i\le m} W_\fp(\mu_i,\nu_i)^\fp.
\end{equation}
Similarly, we have for every $a\in A$ that
\[
\|a\|_\fp^\fp = \sup_{\tau\in\partial_e(T(A))} \tau(|a|^\fp) = \max_{1\le i\le m}\tau_i(|a|^\fp).
\]
Combined with \eqref{it:perturb2}, this implies that for every $f\in B_{\mathrm{diam}(X,\rho)}(C(X))$, we have
\begin{align} \label{eqn:bound1}
\|\varphi'(f)-\psi(f)\|_\fp^\fp &= \left\|\sum_{j=1}^m\psi_j(f\circ h_j)-\psi_j(f)\right\|_\fp^\fp \notag \\
&= \max_{1\le i\le m}\sum_{j=1}^m\tau_i\circ\psi_j(|f\circ h_j-f|^\fp) \notag \\
&\le \max_{1\le i\le m}\tau_i\circ\psi_i(|f\circ h_i-f|^\fp) + \eps \notag \\
&\le \max_{1\le i\le m}\tau_i\circ\psi(|f\circ h_i-f|^\fp) + 2\eps \notag \\
&= \max_{1\le i\le m} \int_X |f(h_i(x))-f(x)|^\fp \, d\nu_i(x) + 2\eps.
\end{align}
It follows that
\begin{align*}
d^\cc_{U,\fp}(\varphi,\psi)^\fp \:&\overset{\mathclap{\scriptstyle\eqref{eqn:bound3}}}{<}\: d^\cc_{U,\fp}(\varphi',\psi)^\fp + \eps\\
\:&\overset{\mathclap{\scriptstyle\eqref{eqn:aa}}}{=}\: \inf_{u\in U(A)} \sup_{f\in\mathrm{Lip}^\cc_1(X,\rho)\cap B_{\mathrm{diam}(X,\rho)}} \|\varphi'(f)-u\psi(f)u^*\|_\fp^\fp + \eps\\
&\le\: \sup_{f\in\mathrm{Lip}^\cc_1(X,\rho)\cap B_{\mathrm{diam}(X,\rho)}} \|\varphi'(f)-\psi(f)\|_\fp^\fp + \eps\\
\:&\overset{\mathclap{\scriptstyle\eqref{eqn:bound1}}}{\le}\: \sup_{f\in\mathrm{Lip}^\cc_1(X,\rho)\cap B_{\mathrm{diam}(X,\rho)}} \max_{1\le i\le m} \int_X |f(h_i(x))-f(x)|^\fp \, d\nu_i(x) + 3\eps\\
&\overset{\mathclap{\scriptstyle\eqref{eqn:lipf}}}{\le}\: \max_{1\le i\le m} \int_X \rho(x,h_i(x))^\fp \, d\nu_i(x) + 3\eps\\
\:&\overset{\mathclap{\scriptstyle\eqref{eqn:transport2i}}}{\le}\: \max_{1\le i\le m} k_{(X,\rho,\fp)}^\fp \cdot W_\fp(\mu_i,\nu_i)^\fp + 4\eps\\
\:&\overset{\mathclap{\scriptstyle\eqref{eqn:bound2}}}{=}\: k_{(X,\rho,\fp)}^\fp \cdot W_\fp(\varphi,\psi)^\fp + 4\eps.
\end{align*}
Since $\eps$ was chosen arbitrarily, this implies that $d^\cc_{U,\fp}(\varphi,\psi) \le k_{(X,\rho,\fp)} \cdot W_\fp(\varphi,\psi)$.
\end{proof}

\begin{remark} \label{rem:bauer}
It is shown in \cite[Theorem 3.1]{Jacelon:2021vc} that for real rank zero $A$ and $\fp=\infty$, Theorem~\ref{thm:transport} (hence also Corollary~\ref{cor:transport}) continues to hold if we relax finiteness of $\partial_e(T(A))$ to $\partial_e(T(A))$ being compact and of finite Lebesgue covering dimension. The idea is to diagonalise relative to orthogonal projections $p_1,\dots,p_m$ that are tracially close to a suitable partition of unity $g_1,\dots,g_m$ defined on $\partial_e(T(A))$, and then proceed in exactly the same way as above with a finite set of extremal traces $\{\tau_i\in g_i^{-1}(\{1\}) \mid 1\le i\le m\}$. As noted in \cite[Theorem 4.12]{Jacelon:2021vc}, this generalisation also works for $\fp<\infty$ when $X$ is an interval. The reason is that in this case, Proposition~\ref{prop:interval} tells us in advance the form of the homeomorphism $h_i$ that transports (a diffuse approximant of) $\mu_{\tau_i\circ\psi}$ to (a diffuse approximant of) $\mu_{\tau_i\circ\varphi}$. We can therefore arrange from the start for our open cover $U_1,\dots,U_m$, to which $g_1,\dots,g_m$ are subordinate, to have the property that $\tau\circ\psi(|f\circ h_i-f|^\fp)$ is close to $\tau_i\circ\psi(|f\circ h_i-f|^\fp)$ for every $\tau\in U_i$ and $f\in B_{\mathrm{diam}(X,\rho)}(C(X))$. In other words, we still have a bound of the form \eqref{eqn:bound1} and we can complete the proof just as above. It is not clear to me whether this works in general.
\end{remark}

\begin{theorem} \label{thm:lowerbound}
Let $(X,\rho)$ be a nonempty compact metric space, let $A$ be a simple, unital $\cs$-algebra with $T(A)\ne\emptyset$, and let $\varphi,\psi\colon C(X) \to A$ be unital $^*$-monomorphisms. Then,
\[
W_\infty(\varphi,\psi) \le d^\rr_{U}(\varphi,\psi).
\]
If $X$ is a subset of the complex plane and $\rho$ is the Euclidean metric, then
\[
W_2(\varphi,\psi) \le d^\cc_{U,2}(\varphi,\psi).
\]
\end{theorem}

\begin{proof}
The $\fp=\infty$ case is \cite[Corollary 3.6]{Jacelon:2021wa}, but we emphasise here that the proof does not depend on the embeddings $\varphi,\psi\colon C(X) \to A$ having commuting images. Instead, there are two key results. The first is \cite[Proposition 3.4]{Jacelon:2021wa}, which establishes that  
\begin{equation} \label{eqn:lipsup}
W_\infty(\varphi,\psi) = \sup_{f\in\mathrm{Lip}^\rr_1(X,\rho)} W_\infty(\varphi(f),\psi(f)).
\end{equation}
The second is \cite[Theorem 2.1(3)]{Hiai:1989aa}, which gives us that $W_\infty(x,y) \le d_U(x,y)$ for self-adjoint elements $x$ and $y$ in a semifinite von Neumann algebra. In particular, this holds in the GNS von Neumann algebra $\pi_\tau(A)''$ associated with any given $\tau\in T(A)$. As in the proof of \cite[Lemma 3.3]{Jacelon:2021wa} (and the proof of Corollary~\ref{cor:normal}), it follows that if $x,y\in A$ are self adjoint, then
$W_\infty(x,y) \le d_U(x,y)$. We conclude as in the proof of \cite[Corollary 3.6]{Jacelon:2021wa} that
\begin{align} \label{eqn:winftychain}
W_\infty(\varphi,\psi) &= \sup_{f\in\mathrm{Lip}^\rr_1(X,\rho)} W_\infty(\varphi(f),\psi(f)) \notag\\
&\le \sup_{f\in\mathrm{Lip}^\rr_1(X,\rho)} d_U(\varphi(f),\psi(f)) \notag\\
&= \sup_{f\in\mathrm{Lip}^\rr_1(X,\rho)} \inf_{u\in U(A)} \|\varphi(f)-u\psi(f)u^*\| \notag\\
&\le \inf_{u\in U(A)} \sup_{f\in\mathrm{Lip}^\rr_1(X,\rho)} \|\varphi(f)-u\psi(f)u^*\| \notag\\
&= d^\rr_U(\varphi,\psi).
\end{align}
We now turn to $\fp\in[1,\infty)$. It is shown in \cite[Proposition 2.3]{Jacelon:2021vc} that $\sup_{f\in\mathrm{Lip}^\rr_1(X,\rho)} W_\fp(f_*\mu,f_*\nu) \le W_\fp(\mu,\nu)$ for every $\mu,\nu\in\prob(X)$. Equality holds if $X$ is an interval, because the inclusion map $X\subseteq\rr$ is contained in $\mathrm{Lip}^\rr_1(X,\rho)$. In fact, the same argument shows that
\[
\sup_{f\in\mathrm{Lip}^\cc_1(X,\rho)} W_\fp(f_*\mu,f_*\nu) \le W_\fp(\mu,\nu)
\]
and that for planar $X$, equality holds via the inclusion $X\subseteq\cc$. For such an $X$, we let $\iota_f \colon C(\sigma(f)) \to C(X)$ denote the functional calculus map and argue as in \cite[Proposition 3.4]{Jacelon:2021wa} to see that
\begin{align*}
W_\fp(\varphi,\psi) &= \sup_{\tau\in T(A)}W_\fp(\mu_{\tau\circ\varphi},\mu_{\tau\circ\psi}) \notag \\
&= \sup_{\tau\in T(A)} \sup_{f\in\mathrm{Lip}^\cc_1(X,\rho)} W_\fp(f_*\mu_{\tau\circ\varphi},f_*\mu_{\tau\circ\psi}) \notag \\
&= \sup_{\tau\in T(A)} \sup_{f\in\mathrm{Lip}^\cc_1(X,\rho)} W_\fp(\mu_{\tau\circ\varphi\circ\iota_f},\mu_{\tau\circ\psi\circ\iota_f}) \notag \\
&= \sup_{f\in\mathrm{Lip}^\cc_1(X,\rho)} \sup_{\tau\in T(A)}  W_\fp(\mu_{\tau\circ\varphi\circ\iota_f},\mu_{\tau\circ\psi\circ\iota_f}) \notag \\
&= \sup_{f\in\mathrm{Lip}^\cc_1(X,\rho)} W_\fp(\varphi(f),\psi(f)).
\end{align*}
So, we have a version of \eqref{eqn:lipsup} for finite $\fp$. When $\fp=2$, we replace \cite[Theorem 2.1(3)]{Hiai:1989aa} by Theorem~\ref{thm:hw} and conclude as in \eqref{eqn:winftychain} that $W_2(\varphi,\psi) \le d^\cc_{U,2}(\varphi,\psi)$.
\end{proof}

\begin{corollary} \label{cor:transport}
Let $A$ be a simple, separable, unital, nuclear, $\js$-stable $\cs$-algebra with $T(A)\ne\emptyset$.
\begin{enumerate}
\item \label{it:main1} If $A$ has real rank zero and finitely many extremal traces, and $X$ is a non-two-dimensional compact, connected Riemannian manifold with torsion-free $K$-theory, then
\[
d^\cc_{U}(\varphi,\psi) = d^\rr_{U}(\varphi,\psi) = W_\infty(\varphi,\psi)
\]
for every pair of unital $^*$-monomorphisms $\varphi,\psi\colon C(X) \to A$ with $K_*(\varphi)=K_*(\psi)$.

\item \label{it:main1a} If $A$ has a unique trace, or has real rank zero and finitely many extremal traces, and $X$ is a nonempty compact, convex subset of Euclidean space, then
\[
d^\cc_{U}(\varphi,\psi) = d^\rr_{U}(\varphi,\psi) = W_\infty(\varphi,\psi)
\]
for every pair of unital $^*$-monomorphisms $\varphi,\psi\colon C(X) \to A$.

\item \label{it:main1b} If $A$ has real rank zero and finitely many extremal traces, and $X$ is a planar circle, then
\[
d^\cc_{U}(\varphi,\psi) = d^\rr_{U}(\varphi,\psi) = W_\infty(\varphi,\psi) \quad \text{and} \quad  W_2(\varphi,\psi) \le d^\cc_{U,2}(\varphi,\psi) \le \tfrac{\pi}{2}\,W_2(\varphi,\psi)
\]
for every pair of unital $^*$-monomorphisms $\varphi,\psi\colon C(X) \to A$ with $K_1(\varphi)=K_1(\psi)$.

\item \label{it:main2} If $A$ has a unique trace, or has real rank zero and finitely many extremal traces, and $X$ is a compact, convex subset of the Euclidean plane, then
\[
d^\cc_{U,2}(\varphi,\psi) = W_2(\varphi,\psi)
\]
for every pair of unital $^*$-monomorphisms $\varphi,\psi\colon C(X) \to A$.
\end{enumerate}
\end{corollary}

\begin{proof}
We apply the theorems above. Note that all of the spaces $X$ being considered have finitely generated, torsion-free $K$-theory.
\begin{enumerate}
\item By Theorem~\ref{thm:lowerbound}, Theorem~\ref{thm:transport} and Theorem~\ref{thm:manifolds} (in that order), we have
\[
W_\infty(\varphi,\psi) \le d^\rr_{U}(\varphi,\psi) \le d^\cc_{U}(\varphi,\psi) \le k_{(X,\rho,\infty)} \cdot W_\infty(\varphi,\psi) = W_\infty(\varphi,\psi)
\]
so equality holds throughout.
\item Since $X$ is contractible, we have that $K_0(C(X))$ is generated by the class of the unit and $K_1(C(X))=0$, so $K_*(\varphi)=K_*(\psi)$ is automatic for any pair of unital $^*$-monomorphisms $\varphi,\psi\colon C(X) \to A$. The equality $d^\cc_{U}(\varphi,\psi) = d^\rr_{U}(\varphi,\psi) = W_\infty(\varphi,\psi)$ now follows exactly as in \eqref{it:main1}, with Theorem~\ref{thm:manifolds} replaced by Theorem~\ref{thm:convex} .
\item Since $K_0(C(S^1))$ is generated by the class of the unit, any two unital $^*$-monomorphisms $\varphi,\psi \colon C(S^1) \to A$ automatically agree on $K_0$. Replacing Theorem~\ref{thm:manifolds} by Theorem~\ref{thm:infinitycircle}, we conclude exactly as in \eqref{it:main1} that $d^\cc_{U}(\varphi,\psi) = d^\rr_{U}(\varphi,\psi) = W_\infty(\varphi,\psi)$, and that
\[
W_2(\varphi,\psi) \le d^\cc_{U,2}(\varphi,\psi) \le k_{(X,\rho,2)} \cdot W_2(\varphi,\psi) \le \tfrac{\pi}{2}\, W_2(\varphi,\psi).
\]
\item Similarly, Theorem~\ref{thm:lowerbound}, Theorem~\ref{thm:transport} and Theorem~\ref{thm:convex} imply that
\[
W_2(\varphi,\psi) \le d^\cc_{U,2}(\varphi,\psi) \le k_{(X,\rho,2)} \cdot W_2(\varphi,\psi) = W_2(\varphi,\psi)
\]
and hence that $d^\cc_{U,2}(\varphi,\psi) = W_2(\varphi,\psi)$ for every pair of unital $^*$-monomorphisms $\varphi,\psi\colon C(X) \to A$. \qedhere
\end{enumerate}
\end{proof}

Now we specialise to planar domains $X$ representing spectra of normal elements. We note that the $\fp=\infty$ case of Corollary~\ref{cor:normal}\eqref{it:main3} was already observed in \cite[Corollary 4.12]{Jacelon:2021wa} under the assumption that $[x]=[y]=0$. Here we are simply pointing out that $K_1$-triviality can be removed as a hypothesis because the transport map provided by Theorem~\ref{thm:infinitycircle} is homotopic to the identity map on the circle.

\begin{corollary} \label{cor:normal}
Let $A$ be a simple, separable, unital, nuclear, $\js$-stable $\cs$-algebra with $T(A)\ne\emptyset$.
\begin{enumerate}
\item \label{it:main3} If $A$ has real rank zero and finitely many extremal traces, and $x,y\in A$ are full-spectrum unitaries with $[x]=[y]$ in $K_1(A)$, then
\[
d_{U}(x,y)=W_\infty(x,y) \quad \text{and} \quad W_2(x,y) \le d_{U,2}(x,y) \le \tfrac{\pi}{2}\,W_2(x,y).
\]
\item \label{it:main4} If $A$ has a unique trace, or has real rank zero and finitely many extremal traces, and $x,y\in A$ are normal elements such that $\sigma(x)=\sigma(y)$ is convex, then
\[
d_{U,2}(x,y)=W_2(x,y).
\]
\end{enumerate}
\end{corollary}

\begin{proof}
Let $x,y\in A$ be normal. As in the proof of Theorem~\ref{thm:lowerbound}, we apply Theorem~\ref{thm:hw} to the GNS von Neumann algebras $\cM_\tau:=\pi_\tau(A)''$ associated with traces $\tau\in T(A)$ to deduce that
\[
W_2(x,y) = \sup_{\tau\in T(A)} W_{2,\tau}(x,y) \le \sup_{\tau\in T(A)} d^{\cM_\tau}_{U,2,\tau}(x,y) \le \sup_{\tau\in T(A)} d^{\pi_\tau(A)}_{U,2,\tau}(x,y) \le d_{U,2}(x,y)
\]
where the superscripts indicate the algebra in which the distance should be measured. Similarly, an application of \cite[Theorem 2.1(5)]{Hiai:1989aa} (just as we applied \cite[Theorem 2.1(3)]{Hiai:1989aa} in the proof of Theorem~\ref{thm:lowerbound}) gives us that $W_\infty(x,y) \le d_U(x,y)$ if $x$ and $y$ are unitaries.

Now we apply Corollary~\ref{cor:transport} to the functional calculus maps $\varphi_x,\varphi_y\colon C(X)\to A$ associated with $x$ and $y$, where $X\subseteq\cc$ is either the circle $S^1$ or is convex. Note that in both cases, we have $K_*(\varphi_x)=K_*(\varphi_y)$. For \eqref{it:main3}, Corollary~\ref{cor:transport}\eqref{it:main1b} then gives
\[
W_\infty(\varphi_x,\varphi_y) = W_\infty(x,y) \le d_U(x,y) \le d^\cc_U(\varphi_x,\varphi_y) = W_\infty(\varphi_x,\varphi_y)
\]
so equality holds throughout. Similarly, we have
\[
W_2(x,y) \le d_{U,2}(x,y) \le d^\cc_{U,2}(\varphi_x,\varphi_y) \le \tfrac{\pi}{2}\, W_2(\varphi_x,\varphi_y) = \tfrac{\pi}{2}\, W_2(x,y).
\]
For \eqref{it:main4}, we have from Corollary~\ref{cor:transport}\eqref{it:main2} that
\[
W_2(\varphi_x,\varphi_y) = W_2(x,y) \le d_{U,2}(x,y) \le d^\cc_{U,2}(\varphi_x,\varphi_y) = W_2(\varphi_x,\varphi_y)
\]
so equality holds throughout.
\end{proof}

\begin{remark} \label{rem:final}
In \cite[Theorem 5.2]{Jacelon:2014aa}, it is shown that $d_U(x,y)=W_\infty(x,y)$ for self-adjoint elements $x$ and $y$ in a simple, separable, unital, exact, $\js$-stable $\cs$-algebra $A$ with $T(A)\ne\emptyset$ provided that $\sigma(x)=\sigma(y)=X$ is connected. Under the additional assumptions on $A$ imposed by Corollary~\ref{cor:normal}\eqref{it:main4}, we also have $d_{U,2}(x,y)=W_2(x,y)$ (and indeed, as highlighted in \cite[Theorem 4.12]{Jacelon:2021vc}, $d_{U,\fp}(x,y)=W_\fp(x,y)$ for any $\fp$). It is a consequence of the classification of embeddings (whether by the Cuntz semigroup, as used in \cite[Theorem 5.2]{Jacelon:2014aa}, or by $K$-theory and traces, as used here and in \cite{Jacelon:2021wa,Jacelon:2021vc}) that these metrics are all topologically equivalent on the set of approximate unitary equivalence classes of self-adjoint elements of $A$ with common connected spectrum $X$. This is in parallel with the fact discussed after Definition~\ref{def:transport} that all of the Wasserstein metrics yield the weak$^*$-topology on the space of faithful Borel probability measures on a connected space.
\end{remark}

Finally, we turn to Corollary~\ref{cor:d}. The goal is to establish a version of Corollary~\ref{cor:transport}\eqref{it:main2} for the `tracial Wasserstein space' $(\js(X),W_2)$ associated with a compact, convex planar domain $(X,\rho)$. Originating in \cite[Section 4.4.4]{Jacelon:2021vc} and \cite[Theorem 4.4]{Jacelon:2022wr} (see also \cite[Corollary 3.7]{Jacelon:2025ab}), the projectionless, classifiable $\cs$-algebra $\js(X)$ is built as an inductive limit of blocks
\[
X_{p_n,q_n} := \{f\in C(X,M_{p_n}\otimes M_{q_n}) \mid f(x_0)\in M_{p_n}\otimes1_{q_n},f(x_1)\in 1_{p_n}\otimes M_{q_n}\}
\]
called generalised dimension drop algebras. Here, $x_0 \ne x_1\in X$ are fixed base points and $(p_n,q_n)$ is a sequence of coprime natural numbers. The connecting maps $\zeta_n\colon X_{p_n,q_n} \to X_{p_{n+1},q_{n+1}}$ are unitarily conjugate to diagonal embeddings $f\mapsto\mathrm{diag}(f\circ\xi_1,\dots,f\circ\xi_{M_n})$, where most of the eigenvalue functions $\xi_i$ are the identity map $X\to X$ (which ensures that $\partial_e(T(\js(X))) \cong \partial_e(T(X_{p_n,q_n})) \cong X$), most of the rest are constant (varying densely over $X$ as the inductive construction progresses to ensure that the limit is simple), and the rest are Lipschitz maps. The proportions guarantee that elements in each compact `nucleus'
\begin{equation} \label{eqn:nucleusdef}
\mathcal{D}(X_{p_n,q_n}) := \{f \in X_{p_n,q_n} \mid \|f\|\le \mathrm{diam}(X),\: \|f(x)-f(y)\|\le\rho(x,y)\:\text{ for every } x,y\in X\}
\end{equation}
are mapped to `tracially Lipschitz' elements in the limit, in the following sense. We identify $T(\js(X))$ with the inverse limit $\varprojlim(\prob(X),T(\zeta_n))$, which, for a given $\fp\in[1,\infty)$, we equip with the limiting $\fp$-Wasserstein metric
\[
\overrightarrow{W}_\fp((\mu_n),(\nu_n)) := \limsup_{n\to\infty}W_\fp(\mu_n,\nu_n).
\]
Then, for any $f\in\mathcal{D}(X_{p_n,q_n})$, the evaluation functional $\widehat{\zeta_{n,\infty}(f)} \colon T(\js(X)) \to \rr$, $\tau \mapsto \tau(\zeta_{n,\infty}(f))$, is Lipschitz with respect to $\overrightarrow{W}_1$, hence with respect to $\overrightarrow{W}_\fp \ge \overrightarrow{W}_1$ for any $\fp$.

By construction, every sequence $(\mu_n)$ in $T(\js(X))$ is weak$^*$-convergent. Sending $(\mu_n)$ to its limit gives, for any $\fp\in[1,\infty)$, an affine isometric isomorphism $(T(\js(X)),\overrightarrow{W}_\fp) \cong (\prob(X),W_\fp)$. This is the metric structure that we use to measure the uniform tracial $\fp$-Wasserstein distance
\[
\overrightarrow{W_\fp}(\varphi,\psi) := \sup_{\tau\in T(A)} W_\fp(\mu_{\tau\circ\varphi},\mu_{\tau\circ\psi})
\]
between unital $^*$-homomorphisms $\varphi$ and $\psi$ from $\js(X)$ into a tracial $\cs$-algebra $A$ (cf.\ \eqref{eqn:tracialwass}). In other words, if $\varphi_n:=\varphi\circ\zeta_{n,\infty}$ and $\psi_n:=\psi\circ\zeta_{n,\infty}$ denote the unital embeddings of $X_{p_n,q_n}$ into $A$ corresponding to the inclusion maps $\zeta_{n,\infty} \colon X_{p_n,q_n} \to \js(X)=\varinjlim(X_{p_k,q_k},\zeta_k)$, then
\begin{equation} \label{eqn:limitwp}
\overrightarrow{W_\fp}(\varphi,\psi) = \sup_{\tau\in T(A)} \overrightarrow{W}_\fp((\mu_{\tau\circ\varphi_n}),(\mu_{\tau\circ\psi_n})) =  \sup_{\tau\in T(A)} \limsup_{n\to\infty}W_\fp(\mu_{\tau\circ\varphi_n},\mu_{\tau\circ\psi_n}).
\end{equation}
We define the $\fp$-norm unitary-orbit distance between $\varphi$ and $\psi$ to be
\begin{equation} \label{eqn:dupdef}
\overrightarrow{d_{U,\fp}}(\varphi,\psi) := \limsup_{n\to\infty} \inf_{u\in\mathcal{U}(A)} \sup_{f\in\mathcal{D}(X_{p_n,q_n})} \|\varphi_n(f)-u\psi_n(f)u^*\|_\fp.
\end{equation}
In other words, $\overrightarrow{d_{U,\fp}}(\varphi,\psi) < \eps$ if and only if, for every large $n$, there is a unitary $u\in A$ such that for every $f\in \mathcal{D}(X_{p_n,q_n})$ and every trace $\tau\in T(A)$, we have $\|\varphi_n(f)-u\psi_n(f)u^*\|_{\fp,\tau}<\eps$.

Our construction does not work if $X=\{\bullet\}$ is a singleton, so we adopt the convention that $\js(\bullet)$ is the Jiang--Su algebra $\js$ \cite{Jiang:1999hb}. Equally, $\js(\bullet)$ is the `quantum intertwining gap' limit of any sequence $(\js(X_k),W_\fp)$ associated with compact, convex domains $X_k$ such that $\mathrm{diam}(X_k)\to0$ (see \cite[Definition 3.11]{Jacelon:2025ab}). This is the right choice for Proposition~\ref{prop:zx}, Theorem~\ref{thm:zx} and Corollary~\ref{cor:zx} as $T(\js)$, and $[\mathrm{Hom}^1(\js,A)]$ for the relevant codomains $A$, are both singletons.

\begin{proposition} \label{prop:zx}
Let $\fp\in[1,\infty]$, let $(X,\rho)$ be a nonempty compact, convex subset of Euclidean space, let $(\js(X),W_\fp)$ be the corresponding tracial $\fp$-Wasserstein space, and let $A$ be a simple, separable, unital, nuclear, $\js$-stable $\cs$-algebra with $T(A)\ne\emptyset$. Then, $\overrightarrow{d_{U,\fp}}$ is a metric on the set $[\mathrm{Hom}^1(\js(X),A)]$ of approximate unitary equivalence classes of unital $^*$-homomorphisms $\js(X)\to A$.
\end{proposition}

\begin{proof}
Symmetry and the triangle inequality follow from Proposition~\eqref{prop:metric}, just as in the proof of Theorem~\ref{thm:cstarmetric}. We will show that two unital $^*$-homomorphisms $\varphi,\psi \colon \js(X) \to A$ are approximately unitarily equivalent if and only if $\overrightarrow{d_{U,\fp}}(\varphi,\psi)=0$.

Since $X$ is contractible, we have from \cite[Proposition 1.16]{Jacelon:2025ab} that each $X_{p_n,q_n}$ (and $\js(X)$ itself) has point-like ordered $K$-theory
\begin{equation} \label{eqn:pointlike}
(K_0(X_{p_n,q_n}),K_0(X_{p_n,q_n})_+,[1_{X_{p_n,q_n}}],K_1(X_{p_n,q_n})) \cong (\zz,\nn,1,0).
\end{equation}
In particular, any two unital $^*$-homomorphisms $\varphi,\psi \colon \js(X) \to A$ automatically agree on $K$-theory. Since $\js(X)$ is simple, these $^*$-homomorphisms are also all embeddings. Every nucleus $\mathcal{D}(X_{p_n,q_n})$ is compact, so it follows as in the proof of Theorem~\ref{thm:cstarmetric} that $\overrightarrow{d_{U,\fp}}(\varphi,\psi)=0$ if $\varphi$ and $\psi$ are approximately unitarily equivalent. For the converse, we extract from \cite[Example 2.15]{Jacelon:2024aa} that
\begin{equation} \label{eqn:nucleus}
\mathcal{D}(X_{p_n,q_n}) \cap (X_{p_n,q_n})_{sa} + \rr \cdot 1 + (X_{p_n,q_n})_0 = \{f \in (X_{p_n,q_n})_{sa} \mid \mathrm{tr} \circ f \in \mathrm{Lip}^\rr_1(X,\rho)\} =: \mathcal{L}_n.
\end{equation}
Here, $(X_{p_n,q_n})_0$ is the set of self-adjoint elements of $X_{p_n,q_n}$ that vanish on all traces, and $\mathrm{tr}$ denotes the unique trace on $M_{p_nq_n}$. This is in fact the defining property of a nucleus (see \cite[Theorem 2.12]{Jacelon:2024aa}). Now let $m\in\nn$, $f\in\mathcal{L}_m$ and $\eps>0$. The construction of $\js(X)$ ensures that there is $K>0$ such that $\zeta_{m,n}(f) \in K\mathcal{L}_n$ for every $n\ge m$. Suppose that $\overrightarrow{d_{U,\fp}}(\varphi,\psi)=0$. Then for every sufficiently large $n$, we deduce as in Theorem~\ref{thm:hadwin} that $|\tau(\varphi_n(g))-\tau(\psi_n(g))| < \eps$ for every $g \in \mathcal{D}(X_{p_n,q_n})$ and every $\tau\in T(A)$. Since $\varphi_n,\psi_n \colon X_{p_n,q_n} \to A$ are unital, it then follows from \eqref{eqn:nucleus} that for every $\tau\in T(A)$,
\[
|\tau(\varphi_m(f))-\tau(\psi_m(f))|  = |\tau(\varphi_n(\zeta_{m,n}(f)))-\tau(\psi_n(\zeta_{m,n}(f)))| < K\eps
\]
and hence, since $\eps$ is arbitrary, that $\tau(\varphi_m(f))=\tau(\psi_m(f))$. Since $\mathcal{L}_m$ has dense linear span in $X_{p_m,q_m}$, this implies that $T(\varphi_m)=T(\psi_m)$. As in the proof of Theorem~\ref{thm:cstarmetric}, we conclude via classification (Theorem~\ref{thm:class1}\eqref{it:class1}) that $\varphi_m$ and $\psi_m$ are approximately unitarily equivalent. Since $m$ is arbitrary, we deduce the same for $\varphi$ and $\psi$.
\end{proof}

\begin{theorem} \label{thm:zx}
Let $X$ be a nonempty compact, convex subset of the Euclidean plane, and let $A$ be a simple, separable, unital, nuclear, $\js$-stable $\cs$-algebra that has a unique trace. Then, $\overrightarrow{d_{U,2}}(\varphi,\psi)=\overrightarrow{W_2}(\varphi,\psi)$ for every pair of unital $^*$-homomorphisms $\varphi,\psi \colon \js(X) \to A$.
\end{theorem}

\begin{proof}
The proof is an adaptation of Corollary~\ref{cor:transport} to the setting of generalised dimension drop algebras (cf.\ \cite[Theorems 6.4 and 6.6]{Jacelon:2025ab}). Let $\varphi,\psi \colon \js(X) \to A$ be unital $^*$-homomorphisms. For each $n\in\nn$, define $\mu_n:=\mu_{\tau\circ\varphi_n}, \nu_n:=\mu_{\tau\circ\psi_n} \in \prob(X)$, where $\tau$ denotes the unique trace of $A$. The map $\iota_n(f):=\mathrm{diag}(f,\dots,f)$ is an embedding of $C(X)$ into the centre of $X_{p_n,q_n}$ such that:
\begin{enumerate}[(i)]
\item \label{it:iota1} $\iota_n\left(\mathrm{Lip}^\cc_1(X,\rho)\cap B_{\mathrm{diam}(X,\rho)}\right) \subseteq \mathcal{D}(X_{p_n,q_n})$;
\item \label{it:iota2} $\iota_n$ induces the identity map on traces under our identifications $T(X_{p_n,q_n}) \cong \prob(X) \cong T(C(X))$.
\end{enumerate}
We therefore have from Theorem~\ref{thm:lowerbound} that
\begin{align*}
\overrightarrow{d_{U,2}}(\varphi,\psi) \:&\overset{\mathclap{\scriptstyle\eqref{eqn:dupdef}}}{=}\: \limsup_{n\to\infty} \inf_{u\in\mathcal{U}(A)} \sup_{f\in\mathcal{D}(X_{p_n,q_n})} \|\varphi_n(f)-u\psi_n(f)u^*\|_2\\
\:&\overset{\mathclap{\scriptstyle\eqref{it:iota1}}}{\ge}\: \limsup_{n\to\infty} \inf_{u\in\mathcal{U}(A)} \sup_{f\in\mathrm{Lip}^\cc_1(X,\rho)\cap B_{\mathrm{diam}(X,\rho)}} \|\varphi_n\circ\iota_n(f)-u\psi_n\circ\iota_n(f)u^*\|_2\\
\:&\overset{\mathclap{\scriptstyle\eqref{eqn:aa}}}{=}\: \limsup_{n\to\infty} d^\cc_{U,2}(\varphi_n\circ\iota_n,\psi_n\circ\iota_n)\\
\:&\overset{\mathclap{\scriptstyle\ref{thm:lowerbound}}}{\ge}\: \limsup_{n\to\infty} W_2(\varphi_n\circ\iota_n,\psi_n\circ\iota_n)\\
\:&\overset{\mathclap{\scriptstyle\eqref{it:iota2}}}{=}\: \limsup_{n\to\infty} W_2(\mu_n,\nu_n)\\
\:&\overset{\mathclap{\scriptstyle\eqref{eqn:limitwp}}}{=}\: \overrightarrow{W_2}(\varphi,\psi).
\end{align*}
For the reverse inequality, we adapt Theorem~\ref{thm:transport}. Let $n\in\nn$ and  $\eps>0$. By Theorem~\ref{thm:convex}, we can find a homeomorphism $h \colon X\to X$ such that $h_*(\nu_n)$ is close to $\mu_n$ and $\|\rho(h(x),x)\|^2_{L^{2}(X,\nu_n)} \le (W_2(\mu_n,\nu_n) + \eps)^2$. The construction outlined in the proof of the theorem also allows us to assume that $h(x_0)=x_0$ and $h(x_1)=x_1$. The map $f \mapsto f\circ h$ therefore defines an automorphism $\alpha_h$ of $X_{p_n,q_n}$. The embedding $\varphi'_n:=\psi_n\circ\alpha_h$ agrees with $\varphi_n$ on $K$-theory and approximately agrees on traces, so by Theorem~\ref{thm:class2}, we can deduce that the maps $\varphi_n,\varphi'_n \colon X_{p_n,q_n} \to A$ are unitarily conjugate within $\eps$ on the compact set $\mathcal{D}(X_{p_n,q_n})$. Again letting $\mathrm{tr}$ denote the unique trace on $M_{p_nq_n}$, we have
\begin{align*}
\sup_{f \in \mathcal{D}(X_{p_n,q_n})} \|\varphi'_n(f) - \psi_n(f)\|^2_2 &= \sup_{f \in \mathcal{D}(X_{p_n,q_n})} \tau\circ\psi_n(|f\circ h - f|^2)\\
&= \sup_{f \in \mathcal{D}(X_{p_n,q_n})} \int_X \mathrm{tr}(|f(h(x)) - f(x)|^2)\, d\nu_n(x)\\
&\le \sup_{f \in \mathcal{D}(X_{p_n,q_n})} \int_X \|f(h(x)) - f(x)\|^2\, d\nu_n(x)\\
&\le \int_X \rho(h(x),x)^2\, d\nu_n(x)\\
&\le (W_2(\mu_n,\nu_n) + \eps)^2.
\end{align*}
Thus, there is a unitary $u\in A$ such that $\|u\varphi_n(f)u^*-\psi_n(f)\|_2  \le W_2(\mu_n,\nu_n) + 2\eps$ for every $f \in \mathcal{D}(X_{p_n,q_n})$. We conclude that $\overrightarrow{d_{U,2}}(\varphi,\psi) \le \overrightarrow{W_2}(\varphi,\psi)$, and hence that $\overrightarrow{d_{U,2}}(\varphi,\psi) = \overrightarrow{W_2}(\varphi,\psi)$.
\end{proof}

Recall from Theorem~\ref{thm:lengthspaces} that for $\fp\in[1,\infty)$, the classical Wasserstein space $(\prob(X),W_\fp)$ over a compact, convex Euclidean domain $X$ is a compact, strictly intrinsic length space: the $\fp$-Wasserstein distance between any two measures is equal to the length of a displacement interpolation from one to the other. Our final result is an immediate consequence of this observation, together with classification (Theorem~\ref{thm:class1}) and Theorem~\ref{thm:zx}.

\begin{corollary} \label{cor:zx}
Let $X$ be a nonempty compact, convex subset of the Euclidean plane, and let $A$ be a simple, separable, unital, nuclear, $\js$-stable $\cs$-algebra with a unique trace. Then, the metric space $([\mathrm{Hom}^1(\js(X),A)],\overrightarrow{d_{U,2}})$ of approximate unitary equivalence classes of unital $^*$-homomorphisms $\js(X)\to A$ is isometric to the compact, strictly intrinsic length space $(\prob(X),W_2)$ of Borel probability measures on $X$.
\end{corollary}

\begin{proof}
The map that sends $[\varphi] \in [\mathrm{Hom}^1(\js(X),A)]$ to $\mu_{\tau\circ\varphi} \in \prob(X)$ is bijective (by Theorem~\ref{thm:class1}\eqref{it:class2}) and isometric (by Theorem~\ref{thm:zx}).
\end{proof}

\end{document}